\documentclass{article}

\usepackage[T1]{fontenc}
\usepackage{lmodern}
\usepackage[utf8x]{inputenx}
\usepackage{framed}
\usepackage{listings}
\usepackage{vmargin}
\usepackage{mathrsfs, mathenv}
\usepackage{amsmath, amsthm, amssymb, amsfonts, amscd}
\usepackage{graphicx}
\usepackage{epstopdf}
\usepackage[svgnames]{xcolor}
\usepackage{hyperref}
\usepackage[capitalise]{cleveref}
\hypersetup{citecolor=blue, colorlinks=true, linkcolor=black}
\usepackage{cite}
\usepackage{pgfplots}
\usepackage{tikz}
\usepackage{color}
\usepackage{lipsum}
\usepackage{autonum}
\usepackage[usestackEOL]{stackengine}
\usepackage{stmaryrd} 
\usepackage[title]{appendix}
\usepackage{overpic}

\theoremstyle{plain}
\newtheorem{theorem}{Theorem}[section]

\newtheorem{lemma}[theorem]{Lemma}
\newtheorem{proposition}[theorem]{Proposition}
\numberwithin{equation}{section}
\newtheorem{algorithm}[theorem]{Algorithm}

\theoremstyle{definition}
\newtheorem{definition}[theorem]{Definition}

\theoremstyle{remark}
\newtheorem{remark}[theorem]{Remark}
\newtheorem{assumption}[theorem]{Assumption}

\ifpdf
  \DeclareGraphicsExtensions{.eps,.pdf,.png,.jpg}
\else
  \DeclareGraphicsExtensions{.eps}
\fi

\usepackage{mathtools}

\usepackage{booktabs}

\usepackage{pgfplots}
\usepackage{tikz}
\usetikzlibrary{patterns,arrows,decorations.pathmorphing,backgrounds,positioning,fit,matrix}
\usepackage[labelfont=bf]{caption}
\usepackage{here}
\usepackage[font=normal]{subcaption}

\usepackage{enumitem}
\setlist[itemize]{leftmargin=.5in}
\setlist[enumerate]{leftmargin=.5in,topsep=3pt,itemsep=3pt,label=(\roman*)}

\newcommand{\email}[1]{\href{#1}{#1}}
\newcommand{\TheTitle}{Sampling Methods for Bayesian Inference Involving Convergent Noisy Approximations of Forward Maps} 
\newcommand{\TheAuthors}{G. Garegnani}
\title{{\TheTitle}}

\author{Giacomo Garegnani\thanks{Institute of Mathematics, \'Ecole Polytechnique F\'ed\'erale de Lausanne, \email{giacomo.garegnani@epfl.ch}}}
\date{}

\usepackage{amsopn}

\newcommand{\abs}[1]{\left\lvert #1 \right\rvert}
\newcommand{\norm}[1]{\left\lVert #1 \right\rVert}

\renewcommand{\phi}{\varphi}
\renewcommand{\theta}{\vartheta}

\newcommand{\iid}{\ensuremath{\stackrel{\text{i.i.d.}}{\sim}}}

\newcommand{\R}{\mathbb{R}}

\newcommand{\defeq}{\coloneqq}

\newcommand{\E}{\operatorname{\mathbb{E}}}
\newcommand{\Eb}[1]{\E\left[#1\right]}

\newcommand{\trace}{\operatorname{tr}}
\newcommand{\Hell}{d_{\mathrm{H}}}
\newcommand{\dd}{\, \mathrm{d}}
\renewcommand{\d}{\mathrm{d}}

\newcommand{\expapp}{\widehat{\exp}}
\newcommand{\V}{\mathbb{V}}
\newcommand{\Vm}[2]{\V^{#1}\left[#2\right]}
\newcommand{\Em}[2]{\E^{#1}\left[#2\right]}
\newcommand{\bias}{\mathrm{Bias}}

\ifpdf
\hypersetup{
	pdftitle={\TheTitle},
	pdfauthor={\TheAuthors}
}
\fi

\definecolor{leg1}{RGB}{0,114,189}
\definecolor{leg2}{RGB}{217,83,25}
\definecolor{leg3}{RGB}{237,177,32}
\definecolor{leg4}{RGB}{126,47,142}
\definecolor{leg5}{RGB}{119,172,48}

\definecolor{leg21}{RGB}{62,38,169}
\definecolor{leg22}{RGB}{46,135,247}
\definecolor{leg23}{RGB}{55,200,151}
\definecolor{leg24}{RGB}{254,195,56}

\crefformat{appendix}{Appendix #2#1#3}

\begin{document}
\maketitle 

\begin{abstract} We present Bayesian techniques for solving inverse problems which involve mean-square convergent random approximations of the forward map. Noisy approximations of the forward map arise in several fields, such as multiscale problems and probabilistic numerical methods. In these fields, a random approximation can enhance the quality or the efficiency of the inference procedure, but entails additional theoretical and computational difficulties due to the randomness of the forward map. A standard technique to address this issue is to combine Monte Carlo averaging with Markov chain Monte Carlo samplers, as for example in the pseudo-marginal Metropolis--Hastings methods. In this paper, we consider mean-square convergent random approximations, and quantify how Monte Carlo errors propagate from the forward map to the solution of the inverse problems. Moreover, we review and describe simple techniques to solve such inverse problems, and compare performances with a series of numerical experiments.
\end{abstract}

\textbf{AMS subject classifications.} 62F15, 65C05, 65N21, 65N75.

\textbf{Keywords.} Bayesian Inference, Randomized Approximations, Monte Carlo Sampling.

\section{Introduction}

In this work, we study inverse problems involving random approximations of a possibly deterministic forward map. The setting we consider arises in many applications. In particular, one can think of random misfit models, in which the likelihood is computed only on a random subset of the data (see e.g. \cite{LMB17,LST18}), or of multiscale models, where fast-scale effects can be modelled as random and one infers an effective slow-scale map from multiscale data \cite{DSW21}. Another possible application is given by probabilistic numerical methods (see e.g. the review papers \cite{OaS19, COS19, HOG15}), which came into prominence in recent years and whose purpose is to quantify numerical errors in a statistical manner, rather than with traditional error bounds. 

We set ourselves in a framework of mean-square convergent approximations, meaning that the random surrogate converges strongly to the true forward map with respect to a discretization parameter $h$ vanishing. Let us remark that discretization, as in the case of random misfits model, could be as well and equivalently interpreted with respect to some growing integer parameter $N$, where $N$ is the number of points (i.e., the computational power) on which the approximation is computed.

In the context of multiscale problems, it is often possible to exploit ergodic properties of the fast-scale component of the forward map in order to extract single-scale surrogates. In particular, the case of forward maps yielding a multiscale diffusion processes, where the whole forward map is intrinsically random, has been considered in a series of works \cite{AGP21,GaZ21,AGZ20,PaS07,PPS09,PPS12}. Inversion of multiscale forward maps with randomization and ensemble techniques has recently been proposed in \cite{DSW21}. In other related works, numerical homogenization (see e.g. \cite{AEE12,PaS08,BLP78,CiD99}) is employed to invert multiscale forward maps involving partial differential equations in a fully deterministic fashion \cite{NPS12,AbD19} or by interpreting (only) the parameter as a random variable in a Bayesian framework \cite{AbD19,AGZ20}.

Probabilistic numerical methods have been developed for a series of diverse numerical tasks, including the numerical solution of linear systems, or the computation of integrals. Most notably, a series of works focused on the probabilistic quantification of approximation errors in the solution of ordinary \cite{Ski92, KeH16, KSH20, TKS19, SSH19, CCC16, SDH14, MaM21, CGS17, LSS19b, TLS18, TZC16, LSS21} and partial differential equations \cite{COS17, COS17b, OCA19, CCC16, CGS17, Owh15, Owh17, OwZ17, RPK17, RPK17b, GFY21}. One of the advantages of probabilistic numerical methods is that they allow ``propagating uncertainty in computational
pipelines'' \cite{OaS19}, where uncertainty is due to numerical discretization. A notable example of such computational pipelines is given by inverse problems, especially in their Bayesian interpretation, for which the beneficial effects of adopting a probabilistic approach has been demonstrated in a series of works \cite{CGS17,AbG20,AbG21,CCC16,COS17,OCA19,COS17b}. 

Employing a randomized approximation of the forward map in the context of Bayesian inverse problem entails additional theoretical and algorithmic difficulties due to the double randomness: both the parameter and the forward map are in this case not deterministic values. A comprehensive study of the implications due to the replacement of a deterministic map by a randomized approximation is presented in \cite{LST18}, where the authors focus on two possible approaches, which lead to a ``marginal'' and a ``sample'' posterior measure, respectively. In particular, the approximation of the true posterior by these two objects is studied extensively, and convergence results are rigorously proved. 

In this work, we consider a framework similar to the one of \cite{LST18} and present sampling-based approximations of the marginal and the sample measures. Indeed, both those measures contain intractable integrals, which we propose here to approximate by means of Monte Carlo approximations. The main contributions of this paper are mainly two:
\begin{enumerate}
	\item We introduce Monte Carlo-based approximations of the marginal and sample approaches to Bayesian inverse problems, and rigorously study their convergence properties towards their corresponding exact posterior;
	\item We present sampling methodologies which should be employed when solving Bayesian inverse problems which involve probabilistic approximations, and assess their performances numerically on test cases.
\end{enumerate}

The outline of the remainder of this paper is as follows. In \cref{sec:Setting} we introduce the theoretical setting, and the probability measures which are the object of our theoretical study. Then, in \cref{sec:Sample} we describe sampling methodologies for both the marginal and the sample approaches introduced in \cite{LST18}. We then present in \cref{sec:Linear} an assessment of the numerical performances of these methods when applied to linear problems, for which closed-form posteriors exist. In \cref{sec:Conv} we present the proof of our convergence results, which are the main theoretical contribution of this work, and finally draw our conclusions in \cref{sec:Conclusion}.

\subsection{Notation}

We denote by $\mathcal P(X)$ the space of probability measures on a measurable space $(X, \mathcal B(X))$. For any measurable function $\phi \colon X \to \R$ and measure $\mu \in \mathcal P(X)$ we write the expectation and variance of $\phi$ as 
\begin{equation}
	\Em{\mu}{\phi} = \int_\Omega \phi(x) \dd\mu(x), \quad \Vm{\mu}{\phi} = \Em{\mu}{\left(\phi - \Em{\mu}{\phi}\right)^2},
\end{equation}
provided $\phi$ is $\mu$-integrable and square $\mu$-integrable, respectively. For an event $B \in \mathcal B(X)$, we say that $B$ occurs $\mu$-a.s. if $\mu(B) = 1$. Given a probability space $(\Omega, \mathcal A, P)$ and a measurable space $(X, \mathcal B(X))$, we call random variable any measurable function $u \colon \Omega \to X$. We employ the acronym i.i.d. for a set of independent and identically distributed random variables $\{u^{(i)}\}_{i=1}^M$. For $p \in [0, \infty]$ and $\mu \in \mathcal P(X)$, we employ the symbol $L^p_\mu(X)$ for the usual Lebesgue space
\begin{equation}
	L^p_\mu(X) \defeq \left\{\phi \colon X \to \R, \int_X \abs{\phi(u)}^p \dd\mu(u) < \infty\right\},
\end{equation} 
with associated norm $\norm{\cdot}_{L^p_\mu(X)}$, and where the choice $p = \infty$ yields the usual $L_\mu^\infty(X)$ space. Let $\nu, \mu \in \mathcal P(X)$ be both absolutely continuous with respect to  a reference measure $\lambda \in \mathcal P(X)$. We then denote by $\Hell(\mu, \nu)$ the Hellinger distance between $\nu$ and $\mu$, i.e.,
\begin{equation}
\Hell(\mu, \nu)^2 = \frac12 \int_X \left(\sqrt{\frac{\d \mu}{\d \lambda}} - \sqrt{\frac{\d \nu}{\d \lambda}}\right)^2 \dd \lambda.
\end{equation}
We recall that the Hellinger distance between two measures $\mu, \nu \in \mathcal P(X)$ bounds the difference between expectations and variances with respect to the same measures (see \cite[Lemma 6.37]{Stu10}), and is equivalent to the total variation distance.

\section{Setting}\label{sec:Setting}

Let $(X, \mathcal B(X), \norm{\cdot}_X)$ and $(Y, \mathcal B(Y), \norm{\cdot}_Y)$ be measurable Banach spaces, and let for simplicity $\dim(Y) < \infty$. Let us moreover introduce a function $\mathcal G\colon X \to Y$ that we call the forward map. We then consider the ill-posed inverse problem
\begin{equation}\label{eq:BIP}
	\text{find } u \in X \text{ given observations } y = \mathcal G(u) + \beta \in Y,
\end{equation}
where $\beta \sim \mathcal N(0, \Gamma)$ is a Gaussian noise and $\Gamma$ is a non-singular covariance on $Y$. We regularize problem \eqref{eq:BIP} by adopting the Bayesian paradigm. In particular, we let $\mu_0 \in \mathcal P(X)$ denote the prior measure, so that formally the posterior $\mu \in \mathcal P(X)$ is given in terms of its Radon--Nykodim derivative by
\begin{equation}\label{eq:Posterior}
	\frac{\d \mu}{\d \mu_0}(u) = \frac{\exp(-\Phi^y(u))}{Z^y}, \qquad Z^y = \int_X \exp(-\Phi^y(u)) \dd \mu_0(u),
\end{equation}
where the potential $\Phi^y$ is given by
\begin{equation}
	\Phi^y(u) = \frac12 \norm{\Gamma^{-1/2}(y - \mathcal G(u))}_Y^2.
\end{equation}
In the following, we drop for economy of notation the dependence of $\Phi_h^y$ and $Z_h^y$ on $y$ and write $\Phi_h$ and $Z_h$. We choose for simplicity to consider Gaussian priors only, in particular $\mu_0 = \mathcal N(m_0, C_0)$, where $m_0 \in X$ and $C_0$ is a trace-class covariance operator on $X$. With this choice and under a growth and a Lipschitz condition on $\mathcal G$ \cite[Assumption 2.7]{Stu10}, the posterior $\mu$ is well defined and given by \eqref{eq:Posterior} \cite{Stu10}. Let us remark that other prior choices are discussed e.g. in \cite{DaS16,Sul17,Hos17,HoN17,DHS12}.
 
Let $(\Omega, \mathcal A)$ be a measurable space, let $h > 0$, and let $\mathcal G_h \colon \Omega \times X \to Y$ be a random forward map approximating $\mathcal G$, in a sense specified below. We denote by $\nu_h$ the distribution of the first argument of $\mathcal G_h$, and assume that the randomness of the two arguments of $\mathcal G_h$ are independent. We measure the quality of the approximation of $\mathcal G$ by $\mathcal G_h$ in terms of the mean-square error, which we define below.

\begin{definition}\label{def:MSOrder} Let $\mathcal G \colon X \to Y$ be a forward map and let $\mathcal G_h \colon \Omega \times X \to Y$ be a randomized approximation of $\mathcal G$. Moreover, let $\mu_0 \in \mathcal P(X)$. We say that $\mathcal G_h$ has mean-square order of convergence $s$ with respect to $\mu_0$ if 
\begin{equation}
	\Em{\nu_h}{\norm{\mathcal G(u) - \mathcal G_h(\cdot,u)}^2_Y}^{1/2} \leq Ch^s, \qquad \mu_0\text{-a.s.},
\end{equation}
for some $s > 0$ and for a positive constant $C$ independent of $h$.
\end{definition}

Let us remark that the definition above depends on the measure $\mu_0$. In the following, we write for simplicity that $\mathcal G_h$ has mean-square order $s$, neglecting the dependence on $\mu_0$. Another possibility would be to consider weak approximations of $\mathcal G$, i.e., forward maps $\mathcal G_h$ satisfying
\begin{equation}
	\norm{\mathcal G(u) - \Em{\nu_h}{\mathcal G_h(\cdot, u)}}_Y \leq Ch^w, \qquad \mu_0\text{-a.s.},
\end{equation}
for some $w > 0$ and $C > 0$ independent of $h$. For the purpose of this work, though, it is necessary to make the stronger assumption that the forward map converges in the mean-square sense.

Let us now replace $\mathcal G$ by $\mathcal G_h$ in \eqref{eq:Posterior} and obtain the random measure $\mu_{h,\mathrm{s}} \colon \Omega \to \mathcal P(X)$ with Radon--Nykodim derivative
\begin{equation}\label{eq:Sample}
	\frac{\d \mu_{h,\mathrm{s}}(\omega)}{\d \mu_0}(u) = \frac{\exp(-\Phi_h(\omega, u))}{Z_h(\omega)}, \qquad Z_h(\omega) = \int_X \exp(-\Phi_h(\omega, u)) \dd \mu_0(u),
\end{equation}
where the approximate potential is the random variable $\Phi_h \colon \Omega \times X\to \R$ defined by
\begin{equation}
	\Phi_h(\omega, u) = \frac12 \norm{\Gamma^{-1/2}(y - \mathcal G_h(\omega,u))}_Y^2.
\end{equation}
In the following, we drop for economy of notation the dependence of the random variables $\mathcal G_h$, $\mu_{h,s}$, $\Phi_h$ and $Z_h$ on $\omega \in \Omega$. Well-posedness of the measure $\mu_{h,\mathrm{s}}$ is shown in \cite{LST18}, where it is called the sample approximation of the posterior $\mu$. The randomization introduced by the forward map $\mathcal G_h$ has to be averaged for practical purposes, which can be achieved in two different ways. First, we consider the marginal measure $\mu_{h,\mathrm{m}} \in \mathcal P(X)$, given by
\begin{equation}\label{eq:Marginal}
	\frac{\d \mu_{h,\mathrm{m}}}{\d \mu_0}(u) = \frac{\Em{\nu_h}{\exp(-\Phi_h(u))}}{\Em{\nu_h}{Z_h}}.
\end{equation}
Second, we consider the averaged measure $\mu_{h,\mathrm{a}} \in \mathcal P(X)$, defined by
\begin{equation}\label{eq:Averaged}
\frac{\d \mu_{h,\mathrm{a}}}{\d \mu_0}(u) = \Em{\nu_h}{\frac{\exp(-\Phi_h(u))}{Z_h}}.
\end{equation}
The marginal and the averaged measures are both approximations of the sample measure $\mu_{h,\mathrm{s}}$, but are different in spirit. Indeed, for the former we compute averages over the randomization of $\mathcal G_h$ for any $u \in X$. For the latter, instead, we first fix $\omega \in \Omega$, compute the posterior associated to $\mathcal G_h(\omega, \cdot)$ on the whole $X$, and then average over $\Omega$. 

Both the marginal measure $\mu_{h,\mathrm{m}}$ and the averaged measure $\mu_{h,\mathrm{a}}$ cannot be employed directly. Indeed, the expectation $\E^{\nu_h}$ is, in most cases, intractable. It is therefore natural to employ Monte Carlo integration to obtain approximations of both these measures. Given a positive integer $M$ and a i.i.d. sample $\{\Phi_h^{(i)}\}_{i=1}^M$ such that $\Phi_h^{(1)} \sim \Phi_h$, we introduce the averaged quantities
\begin{equation}
\begin{aligned}
	&\exp^M\left(-\Phi_h(u)\right) \defeq \frac1M\sum_{i=1}^M \exp\left(-\Phi_h^{(i)}(u)\right),\\
	&Z_h^M \defeq \frac1M \sum_{i=1}^M Z_h^{(i)}, \qquad Z_h^{(i)} \defeq \int_X \exp\left(-\Phi_h^{(i)}\right) \dd \mu_0(u).
\end{aligned}
\end{equation}
We can now give the definition of the probability measures we consider in this work.
\begin{definition}\label{def:MCApprox} With the notation above, we define the Monte Carlo approximation $\mu_{h,\mathrm{m}}^M \in \mathcal P(X)$ of the marginal measure as 
	\begin{equation}\label{eq:MargMC}
		\frac{\d \mu_{h,\mathrm{m}}^M}{\d \mu_0}(u) = \frac{\exp^M\left(-\Phi_h(u)\right)}{Z_h^M}.
	\end{equation}
	Moreover, we define the Monte Carlo approximation $\mu_{h,\mathrm{a}}^M \in \mathcal P(X)$ of the averaged measure as
	\begin{equation}
		\mu_{h,\mathrm{a}}^M(\d u) = \frac1M \sum_{i=1}^M \mu_{h,\mathrm{s}}^{(i)}(\d u), \qquad 
		\frac{\d \mu_{h,\mathrm{s}}^{(i)}}{\d \mu_0}(u) = \frac{\exp\left(-\Phi_h^{(i)}(u)\right)}{Z_h^{(i)}}.
	\end{equation}
\end{definition}

\begin{remark} Both the Monte Carlo approximations $\mu_{h,\mathrm{m}}^M$ and $\mu_{h,\mathrm{a}}^M$ are random measures due to the randomness of the sample. In the following, we adopt the slight abuse of notation of denoting by $\E^{\nu_h}$ the expectation with respect to sample draws. 
\end{remark}

\section{Statement of Convergence Results}

Before stating the convergence results we prove in this work, let us introduce a working assumption.
\begin{assumption}\label{as:Marginal} Defining $\exp^M(\Phi_h)$ equivalently to $\exp^M(-\Phi_h)$, the potential $\Phi_h$ and the normalizing constants $Z$, $Z_h$ and $Z_h^M$ satisfy
	\begin{enumerate}
		\item\label{it:as_1} $\min\left\{\Em{\nu_h}{\norm{\exp(\Phi_h)}}_{L^\infty_{\mu_0}(X)}, \norm{\exp^M(\Phi_h)}_{L^\infty_{\mu_0}(X)}\right\} \leq C_1$, $\nu_h$-a.s.,
		\item $C_2^{-1} \leq Z \leq C_2$,
		\item\label{it:as_3} $C_3^{-1} \leq Z_h \leq C_3$, $\nu_h$-a.s.,
		\item\label{it:as_4} $C_4^{-1} \leq Z_h^M \leq C_4$, $\nu_h$-a.s.,
	\end{enumerate}
	for positive constants $\{C_i\}_{i=1}^4$ independent of $h$ and $M$. 
\end{assumption}
Let $\mathcal G_h$ have mean-square order of convergence $s$ with respect to the prior $\mu_0$. Then, it is possible to show under \cref{as:Marginal} that the measures $\mu_{h,\mathrm{m}}$ and $\mu_{h,\mathrm{a}}$ are good approximations of the true posterior $\mu$, in the sense that the Hellinger distance converges with order $s$ with respect to the discretization parameter $h \to 0$ \cite{LST18}. In this paper, we are interested in extending the results of \cite{LST18} to the approximation of the deterministic, but intractable, probability measures $\mu_{h,\mathrm{m}}$ and $\mu_{h,\mathrm{a}}$ by their Monte Carlo approximations given in \cref{def:MCApprox}. In particular, we prove in \cref{sec:Conv} the following convergence results.

\begin{theorem}\label{thm:ConvMarginal} Let $\mu_{h,\mathrm{m}}$ be the marginal posterior measure of \eqref{eq:Marginal} and let its Monte Carlo approximation $\mu_{h,\mathrm{m}}^M$ be given in \cref{def:MCApprox}. If the approximate forward map $\mathcal G_h$ has mean-square order of convergence $s$ with respect to the prior $\mu_0$ and \cref{as:Marginal} holds, then
	\begin{equation}
		\sqrt{\Em{\nu_h}{\Hell\left(\mu_{h,\mathrm{m}}, \mu_{h,\mathrm{m}}^M\right)^2}} \leq C\frac{h^s}{\sqrt{M}},
	\end{equation}
	for a positive constant $C$ independent of $h$ and $M$.
\end{theorem}

\begin{theorem}\label{thm:ConvSample} Let $\mu_{h,\mathrm{a}}$ be the averaged posterior measure of \eqref{eq:Averaged} and let its Monte Carlo approximation $\mu_{h,\mathrm{a}}^M$ be given in \cref{def:MCApprox}. If the approximate forward map $\mathcal G_h$ has mean-square order of convergence $s$ with respect to the prior $\mu_0$ and \cref{as:Marginal} holds, then
	\begin{equation}
		\sqrt{\Em{\nu_h}{\Hell\left(\mu_{h,\mathrm{a}}, \mu_{h,\mathrm{a}}^M\right)^2}} \leq C\frac{h^s}{\sqrt{M}},
	\end{equation}
	for a positive constant $C$ independent of $h$ and $M$. 
\end{theorem} 

Two remarks on the theorems above are due.

\begin{remark} Let us consider \cref{thm:ConvMarginal}, i.e., the convergence result for the marginal measure. Employing the results of \cite{LST18}, and by the triangle inequality, we obtain
\begin{equation}
	\sqrt{\Em{\nu_h}{\Hell\left(\mu, \mu_{h,\mathrm{m}}^M\right)^2}} \leq Ch^s\left(1 + \frac{1}{\sqrt{M}}\right).
\end{equation}
Hence, a priori and if $h$ is small enough, the measure $\mu_{h,\mathrm{m}}^M$ is a good approximation of the true posterior, regardless of the number of samples $M$ that is taken to approximate the expectation with respect to $\nu_h$ appearing in the marginal measure. The same consideration holds for the averaged measure. 
\end{remark}
\begin{remark} \cref{as:Marginal} could be relaxed by considering in \ref{it:as_1} a $L^p_{\mu_0}(X)$ norm for $p < \infty$, instead of the $L_{\mu_0}^\infty$ norm. Moreover, in \ref{it:as_3} and \ref{it:as_4}, we could have imposed boundedness in the $L^{q}_{\nu_h}(\Omega)$ norm for the random normalizing constants, with $q < \infty$, instead of a.s. boundedness. Results similar to \cref{thm:ConvMarginal,thm:ConvSample} would still hold, but we chose the bounded setting for clarity in the statement of the results and in their proofs.
\end{remark}

\section{Sampling from Approximate Posteriors}\label{sec:Sample}

In this section, we explore methods to sample from the posterior distributions presented above, and therefore solve approximately the inverse problem. In particular, let $\psi\colon X \to \R$ be a target function, and let us consider the task of approximating $\Psi \defeq \Em{\mu}{\psi}$, where $\mu$ is the posterior distribution given in \eqref{eq:Posterior}. We assume here that $\dim(X) < \infty$, and argue that this is not restrictive for practical implementation. Indeed, infinite-dimensional problems have to be approximated in finite-dimensional subspaces for practical purposes, which can be achieved e.g. by means of a Karhunen--Loève expansion (see e.g. \cite{Stu10}). The high-dimensionality of the unknown $u \in X$ hints to the use of a Monte Carlo technique, and specifically the approximation on $N$ equally weighted integration points
\begin{equation}\label{eq:MonteCarlo}
	\Psi \approx \widehat \Psi^N = \frac1N \sum_{i=1}^N \psi\left(u^{(i)}\right), \qquad \{u^{(i)}\}_{i=1}^N \sim \mu. 
\end{equation}
Since the problem \eqref{eq:BIP} is multi-dimensional and the normalization constant $Z$ in \eqref{eq:Posterior} is unknown, the sample $\{u^{(i)}\}_{i=1}^N$ in \eqref{eq:MonteCarlo} can be obtained employing a Markov chain Monte Carlo (MCMC) algorithm, such as the random walk Metropolis--Hastings (RWMH), which proceeds as described by \cref{alg:RWMH}. 

\begin{algorithm}[RWMH]\label{alg:RWMH} Let $Q = \mathcal N(0, C_Q)$ be a Gaussian proposal distribution on $X$ and $u^{(0)} \sim \mu_0$ where the prior is the Gaussian $\mu_0 = \mathcal N(0, C_0)$. For $i = 1, 2, \ldots, N$, generate  $\{u^{(i)}\}_{i=1}^N$ as
	\begin{enumerate} 
		\item Sample $\Delta u^{(i)} \sim Q$ and set $\widehat u^{(i)} \sim u^{(i-1)} + \Delta u^{(i)}$;
		\item Set $u^{(i)} = \widehat u^{(i)}$ with probability $\alpha$, and $u^{(i)} = u^{(i-1)}$ with probability $1-\alpha$, where
		\begin{equation}\label{eq:BIP_AcceptProb}
			\begin{aligned}
				\alpha &= \min\left\{\widehat \alpha, 1\right\}, \\
				\widehat \alpha &= \exp\left(-\Phi(\widehat u^{(i)}) + \Phi(u^{(i-1)}) -\frac12 \left(\widehat u^{(i)}, C_0^{-1} \widehat u^{(i)}\right) +\frac12 \left(u^{(i-1)}, C_0^{-1} u^{(i-1)}\right)\right).
			\end{aligned}
		\end{equation}
	\end{enumerate} 
\end{algorithm}

It is then known that the Markov chain $\{u^{(i)}\}_{i\geq 1}$ admits the posterior $\mu$ given in \eqref{eq:Posterior} as an invariant measure, and that therefore the RWMH eventually yields samples from the posterior distribution. The RWMH, or other similar sampling techniques, cannot be directly applied to the sample posterior distribution given in \eqref{eq:Sample}. Indeed, the likelihood function $\exp(-\Phi_h(u))$ computed with the probabilistic forward map $\mathcal G_h$ is a random variable and therefore intractable. In the remainder of this section, we review standard techniques which can be employed in this case.

\subsection{The Marginal Measure}

The marginal approximation $\mu_{h,\mathrm{m}}$ given in \eqref{eq:Marginal} is deterministic, but intractable due to the possibly high-dimensional integral $\Em{\nu_h}{\exp(-\Phi_h)}$. We now present two techniques to sample from $\mu_{h,\mathrm{m}}$: the pseudo-marginal Metropolis Hastings (PMMH) \cite{AnR09} and the Monte Carlo within Metropolis (MCwM) \cite{MLR16,AFE16,MRS20}, highlighting their respective advantages and disadvantages.

The PMMH yields an exact sample from the marginal measure by extending the state space $X$ to include all the random sources which are necessary to evaluate the probabilistic forward map. In practice, the algorithm is equivalent to the RWMH modulo a replacement of the unnormalized likelihood $\exp(-\Phi)$ by its Monte Carlo estimator $\exp^M(-\Phi_h)$. Recycling the value of the Monte Carlo estimator is indeed equivalent to making the Markov chain advance on the extend state space. Since the Monte Carlo estimator is unbiased, it is possible to prove (see \cite{AnR09}) that the marginal on $X$ of the measure targeted by the PMMH is exactly $\mu_{h,\mathrm{m}}$. For clarity, we include the full pseudo-code for the PMMH in \cref{alg:PMMH}.
  
\begin{algorithm}[PMMH]\label{alg:PMMH} Let $Q = \mathcal N(0, C_Q)$ be a Gaussian proposal distribution on $X$ and $u^{(0)} \sim \mu_0$ where the prior is the Gaussian $\mu_0 = \mathcal N(0, C_0)$. For $i = 1, 2, \ldots, N$, generate  $\{u^{(i)}\}_{i=1}^N$ as
	\begin{enumerate}
		\item Sample $\Delta u^{(i)} \sim Q$, set $\widehat u^{(i)} \sim u^{(i-1)} + \Delta u^{(i)}$;
		\item Compute $\exp^M(-\Phi(\widehat u^{(i)}))$ following \cref{def:MCApprox};
		\item Set $u^{(i)} = \widehat u^{(i)}$ with probability $\alpha$, and $u^{(i)} = u^{(i-1)}$ with probability $1-\alpha$, where
		\begin{equation}\label{eq:PMMH_AcceptProb}
			\begin{aligned}
				\alpha &= \min\left\{\widehat \alpha, 1\right\}, \\
				\widehat \alpha &= 
				\begin{aligned}[t]
					&\frac{\exp^M(-\Phi(\widehat u^{(i)}))}{\exp^M(-\Phi(u^{(i-1)}))} \exp\left(-\frac12 \left(\widehat u^{(i)}, C_0^{-1} \widehat u^{(i)}\right) +\frac12 \left(u^{(i-1)}, C_0^{-1} u^{(i-1)}\right)\right).
				\end{aligned}
			\end{aligned}
		\end{equation}
	\end{enumerate}
\end{algorithm}

The advantage of the PMMH is clearly that it targets exactly the posterior. The disadvantages, instead, are mainly two. The first, and clearest, is that the computational cost needed to obtain $N$ samples from the marginal posterior $\mu_{h,\mathrm{m}}$ with the PMMH is equivalent to $N\cdot M$ evaluations of the forward map. The second disadvantage is that the PMMH produces, in some situations, badly-behaved Markov chains. Indeed, if the Monte Carlo estimator of the unnormalized likelihood has a variance which is large with respect to the noise on the observations, then it is likely to observe a sticky behavior of the Markov chain, which may remain stuck in areas which are unlikely under $\mu_{h,\mathrm{m}}$. We remark that this effect is amplified in case the dimension of the data or of the parameter space is large. This second issue can be solved by increasing the number of samples $M$ in the computation of $\exp^M(-\Phi_h)$, and thus by reducing the variance of the estimator. Nevertheless, the value of $M$ required to obtain a well-behaved sample could be extremely high and lead to a high computational cost. 

The MCwM, whose pseudo-code is given in \cref{alg:MCwM}, is designed to solve the second issue presented above.

\begin{algorithm}[MCwM]\label{alg:MCwM} Let $Q = \mathcal N(0, C_Q)$ be a Gaussian proposal distribution on $X$ and $u^{(0)} \sim \mu_0$ where the prior is the Gaussian $\mu_0 = \mathcal N(0, C_0)$. For $i = 1, 2, \ldots, N$, generate  $\{u^{(i)}\}_{i=1}^N$ as
	\begin{enumerate}
		\item Sample $\Delta u^{(i)} \sim Q$, set $\widehat u^{(i)} \sim u^{(i-1)} + \Delta u^{(i)}$;
		\item\label{it:MCwM} Compute $\exp^M(-\Phi(\widehat u^{(i)}))$ and recompute $\exp^M(-\Phi(u^{(i-1)}))$ following \cref{def:MCApprox};
		\item Set $u^{(i)} = \widehat u^{(i)}$ with probability $\alpha$, and $u^{(i)} = u^{(i-1)}$ with probability $1-\alpha$, where
		\begin{equation}\label{eq:MCwM_AcceptProb}
			\begin{aligned}
				\alpha &= \min\left\{\widehat \alpha, 1\right\}, \\
				\widehat \alpha &= 
				\begin{aligned}[t]
					&\frac{\exp^M(-\Phi(\widehat u^{(i)}))}{\exp^M(-\Phi(u^{(i-1)}))} \exp\left(-\frac12 \left(\widehat u^{(i)}, C_0^{-1} \widehat u^{(i)}\right) +\frac12 \left(u^{(i-1)}, C_0^{-1} u^{(i-1)}\right)\right).
				\end{aligned}
			\end{aligned}
		\end{equation}
	\end{enumerate}
\end{algorithm}

We remark that the only difference between the PMMH and the MCwM consists in point \ref{it:MCwM} of both algorithms, where we compute the Monte Carlo estimator of the unnormalized likelihood. Indeed, in the former one recycles the value of the estimator, thus implicitly building a Markov chain on an extended state space, whereas in the second the estimator of the likelihood is computed for both the proposed sample and the previous state of the Markov chain on $X$. In this way, one avoids completely the sticky behavior of the PMMH and the sample produced by the MCwM is of good quality, provided a well-informed choice for the proposal distribution $Q$. The main downside of the MCwM, when compared to the PMMH, is that it does not target the exact marginal posterior $\mu_{h,\mathrm{m}}$, but a perturbed version of it. The distance between the perturbed measure targeted by the MCwM and the true marginal posterior can be quantified employing the tools of \cite{MLR16}. 

Comparing the computational cost that is needed to obtain a representative sample from $\mu_{h,\mathrm{m}}$ employing the PMMH and the MCwM is not a simple task. Let us first notice that for fixed $N$ and $M$ a run of the MCwM yields twice the computational cost as the PMMH, as the estimator of the unnormalized likelihood has to be recomputed for the previous state of the chain. For a small value of $M$, though, the acceptance ratio of the PMMH could be very close to zero, so that the sample resulting from the algorithm is not meaningful in any sense. Conversely, the MCwM outputs a meaningful sample for any value of $M$, drawn, though, from a measure that is close to the true marginal posterior only for large values of $M$. A numerical assessment of this trade-off between good quality of the sample vs the accuracy of the targeted measure is presented in \cref{sec:Linear}.

\subsection{The Averaged Measure}

Sampling form the averaged measure $\mu_{h,\mathrm{a}}$ entails less difficulties than the marginal measure, at least in the design of an algorithm which samples exactly from the posterior. Indeed, let us recall that the Monte Carlo approximation $\mu_{h,\mathrm{a}}^M$ of the averaged measure reads
\begin{equation}
	\mu_{h,\mathrm{a}}^M(\d u) = \frac1M \sum_{i=1}^M \mu_{h,\mathrm{s}}^{(i)}(\d u), \qquad \frac{\d \mu_{h,\mathrm{s}}^{(i)}}{\d \mu_0}(u) = \frac{\exp\left(-\Phi_h^{(i)}(u)\right)}{Z_h^{(i)}},
\end{equation}
where $Z_h^{(i)}$ are for $i = 1, \ldots, M$ the corresponding normalization constants. Approximate sampling from the measure $\mu_{h,\mathrm{a}}$ can therefore be achieved by employing a Metropolis within Monte Carlo (MwMC) approach, where one runs a independent RWMH chain for each realization of the randomized forward map. The MwMC approach is summarized by \cref{alg:MwMC}.

\begin{algorithm}[MwMC]\label{alg:MwMC} Let $Q = \mathcal N(0, C_Q)$ be a Gaussian proposal distribution on $X$ and $u^{(0)} \sim \mu_0$ where the prior is the Gaussian $\mu_0 = \mathcal N(0, C_0)$. For $i = 1, 2, \ldots, M$, 
	\begin{enumerate}
		\item\label{it:MwMC} Generate a random forward map $\mathcal G_h^{(i)}$ with corresponding potential $\Phi_h^{(i)}$;
		\item\label{it:MwMC_2} Obtain a sample $\{u_h^{(i,j)}\}_{j=1}^N$ from the posterior $\mu_{h,\mathrm{s}}^{(i)}$ employing \cref{alg:RWMH} with proposal $Q$, initial value $u^{(0)}$.
		\item Reassemble the sample as $\{u_h^{(i)}\}_{i=1}^{N\cdot M} = \bigcup_{i=1}^M \{u_h^{(i,j)}\}_{j=1}^N$.
	\end{enumerate}
\end{algorithm}

Let us comment on this algorithm. First, we remark that at termination we have a sample of size $N\cdot M$, instead of a sample of size $N$ as for the PMMH and the MCwM, for a comparable cost of $N \cdot M$ evaluations of the forward map. Moreover, due to the unbiasedness of the Monte Carlo estimator, the MwMC yields unbiased estimates of the quantity of interest for any given $M$. Nevertheless, for each realization of the forward map in step \ref{it:MwMC} of \cref{alg:MwMC} the samples are clustered in the posterior corresponding to the forward map $\mathcal G_h^{(i)}$, and are not representative of the whole averaged measure $\mu_{h,\mathrm{a}}$. If the forward map $\mathcal G_h$ is mean-square convergent, though, and $h$ is chosen small, then the posteriors targeted in step \ref{it:MwMC_2} are close to each other, and all close to the true posterior $\mu$. Hence, in this case, it may be unnecessary to choose $M$ large in order to obtain a sample that covers well the posterior $\mu_{h,\mathrm{a}}$. Numerical assessments of the performances of the MwMC are presented in \cref{sec:Linear}.

\section{Numerical Assessment: The Linear Case}\label{sec:Linear}

In this section, we present a numerical assessment of the Monte Carlo sampling strategies presented in \cref{sec:Sample} based on a linear test case, for which the marginal and the averaged posterior distributions are computable explicitly. In particular, let $m$ and $d$ be positive integers and let us consider the inverse problem
\begin{equation}\label{eq:LinGaussIP}
\text{find } u \in \R^d \text{ given observations } y = Au + \beta \in \R^m,
\end{equation}
where $A$ is a matrix in $\R^{m \times d}$, and where $\beta \sim \mathcal N(0, \Gamma)$ is Gaussian random variable with $\Gamma$ a non-singular covariance matrix on $\R^m$. Fixing a Gaussian prior $\mu_0 = \mathcal N(m_0, C_0)$ on the unknown yields in this case a Gaussian posterior $\mu = \mathcal N(m, C)$, where the posterior precision matrix and mean are given by
\begin{equation}\label{eq:LinGaussPost}
\begin{aligned}
C^{-1} &= A^\top \Gamma^{-1} A + C_0^{-1}, \\
m &= C\left(A^\top \Gamma^{-1} y + C_0^{-1} m_0\right),
\end{aligned}
\end{equation}
as shown in \cite[Example 6.23]{Stu10}. With the notation introduced above, we therefore have in this case $\mathcal G(u) = Au$. Let us assume that the matrix $A$ is known only up to an additive perturbation. In particular, we assume there exist $h > 0$ and a matrix $P \in \R^{m\times d}$ whose norm is bounded from above independently of $h$, such that the matrix $A_h = A + hP$ is known. Depending on the value $h > 0$, employing the matrix $A_h$ to solve the inverse problem \eqref{eq:LinGaussIP} would yield wrong and overconfident posterior distributions on the unknown $u$. 

Let us now consider the randomized forward map $\mathcal G_h \colon \omega \times u \mapsto A_hu + h\xi(\omega)$, where $\xi \sim \mathcal N(0, Q)$ for a a symmetric positive semi-definite covariance $Q \in \R^{m\times m}$ whose norm is bounded independently of $h$. For this forward map, the triangle inequality with respect to the $L^2_{\nu_h}(\Omega)$ norm yields
\begin{equation}
\Em{\nu_h}{\norm{\mathcal G_h(u) - \mathcal G(u)}_2^2}^{1/2} \leq h \left(\Eb{\norm{\xi}^2_2}^{1/2} + \norm{Pu}_2\right) = h \left(\sqrt{\trace(Q)} + \norm{Pu}_2\right),
\end{equation}
where $\trace(\cdot)$ is the trace operator and $\mu_0$-a.s. with respect to $u$. Hence, the forward map $\mathcal G_h$ is of first mean-square order of convergence with respect to $h$ in the sense of \cref{def:MSOrder}. We can rewrite the observation model replacing $\mathcal G$ with $\mathcal G_h$ as
\begin{equation}
A_h u + \beta = y - h \xi,
\end{equation}
which shows that the sample posterior $\mu_{h,\mathrm{s}}$ associated to $\mathcal G_h$ is simply the random Gaussian measure $\mu_{h,\mathrm{s}} = \mathcal N(m_{h,\mathrm{s}}, C_{h,\mathrm{s}})$, where
\begin{equation}\label{eq:LinGaussSample}
\begin{aligned}
C_{h,\mathrm{s}}^{-1} &= A_h^\top \Gamma^{-1} A_h + C_0^{-1}, \\
m_{h,\mathrm{s}} &= C_{h,\mathrm{s}}\left(A_h^\top \Gamma^{-1} (y - h\xi) + C_0^{-1} m_0\right).
\end{aligned}
\end{equation}
In the following, we give explicit expressions for the marginal and averaged measures defined in \cref{def:MCApprox} and present numerical experiments for the Monte Carlo techniques given in \cref{sec:Sample}.

\subsection{The Marginal Measure}

We first consider the marginal approximation $\mu_{h,\mathrm{m}}$, which can be computed explicitly as shown by the following result.

\begin{proposition}\label{prop:LinGaussMarg} With the notation introduced above, the marginal posterior on $\R^d$ associated to the randomized forward map $\mathcal G_h\colon \omega \times u \mapsto A_hu + h\xi(\omega)$, is given by $\mu_{h,\mathrm{m}} = \mathcal N(m_{h,\mathrm{m}}, C_{h,\mathrm{m}})$, where
	\begin{equation}
		\begin{aligned}
			C_{h,\mathrm{m}}^{-1} &= A_h^\top \Gamma_h^{-1} A_h + C_0^{-1}, \\
			m_{h,\mathrm{m}} &= C_{h,\mathrm{m}}\left(A_h^\top \Gamma_h^{-1} y + C_0^{-1} m_0\right),
		\end{aligned}
	\end{equation}
	where $\Gamma_h \defeq \Gamma + h^2 Q$ and where we recall that $\xi \sim \mathcal N(0, Q)$.
\end{proposition}
\begin{proof} We use the symbol $\propto$ to denote equality up to a proportionality constant independent of $u$. It holds
	\begin{equation}
		\Em{\nu_h}{\exp(-\Phi_h(u))} \propto \int_{\R^m} \exp\left(-\frac12 \norm{\Gamma^{-1/2}(y - A_hu - h\xi))}^2_2 - \frac12 \norm{Q^{-1/2}\xi}_2^2\right) \dd \xi.
	\end{equation}
	Completing the square inside the integral to get a Gaussian density and algebraic simplifications yields
	\begin{equation}
	\begin{aligned}
		\Em{\nu_h}{\exp(-\Phi_h(u))} \propto \exp\left(-\frac12\norm{\Gamma_h^{-1/2}(y - A_hu)}_2^2\right), 
	\end{aligned}
	\end{equation}
	where $\Gamma_h$ is given in the statement above. Hence, 
	\begin{equation}
		\frac{\d \mu_{h,\mathrm{m}}}{\d\mu} = \frac1{\Em{\nu_h}{Z_h}} \exp\left(-\frac12\norm{\Gamma_h^{-1/2}(y - A_hu)}_2^2\right).
	\end{equation}
	Since $\Em{\nu_h}{Z_h}$ is the normalization constant and hence independent of $u$, this proves the desired result by \cite[Example 6.23]{Stu10}. 
\end{proof}

\begin{remark} The same result of \cref{prop:LinGaussMarg} could be obtained by rewriting the observation model as
	\begin{equation}
		y = A_hu + \beta_h, \qquad \beta_h \defeq \beta + h\xi \sim \mathcal N(0, \Gamma_h),
	\end{equation}
	i.e., by modifying the noise to take into account the randomization of the approximated forward map.
\end{remark}

\subsubsection*{Numerical Experiment}

\begin{figure}[t]
	\centering
	\begin{tabular}{c@{\hskip 1cm}c}
		\multicolumn{2}{c}{\includegraphics[]{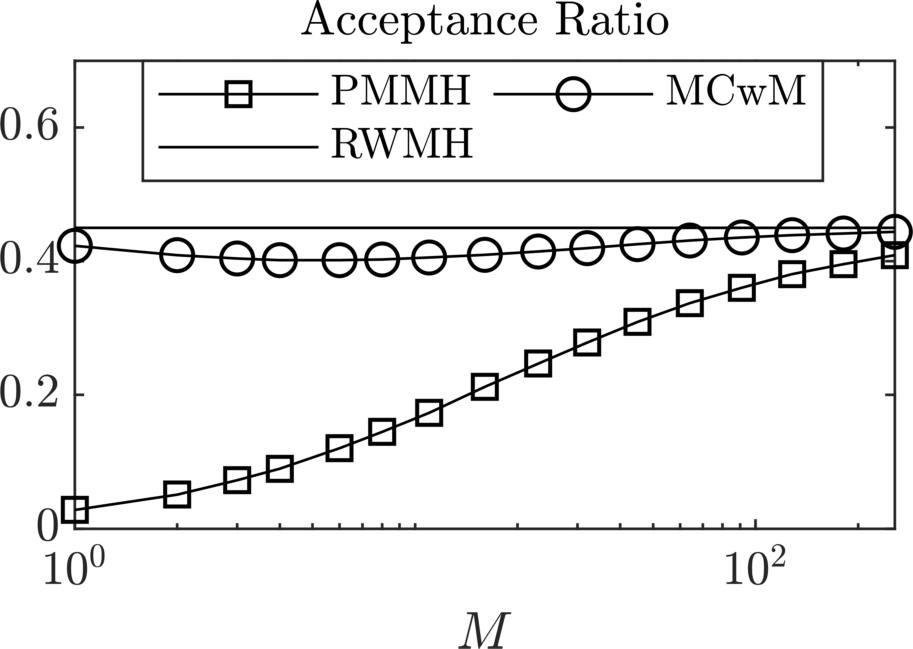}} \vspace{0.3cm} \\
		\includegraphics[]{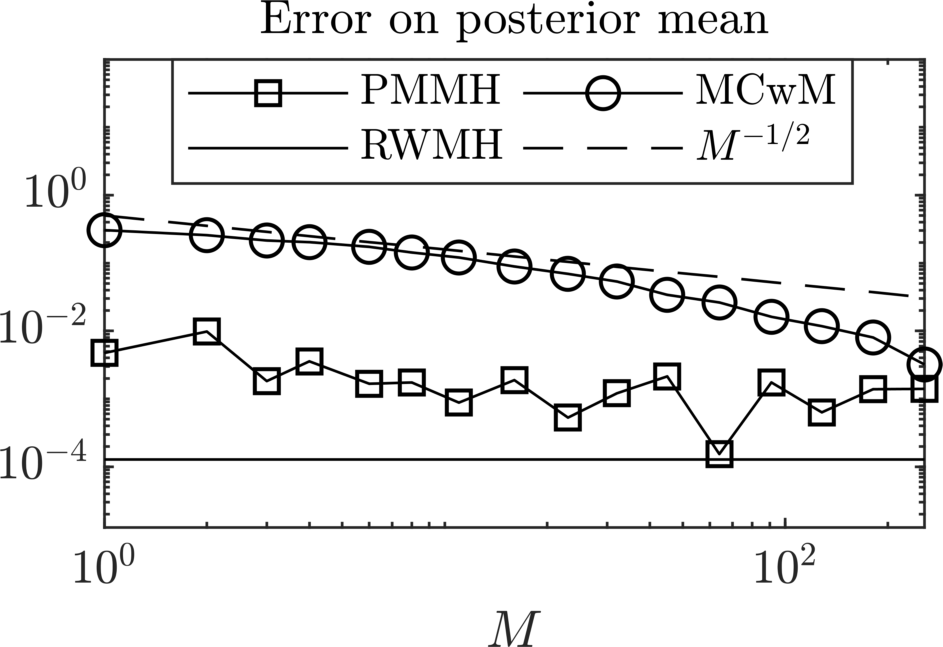} & \includegraphics[]{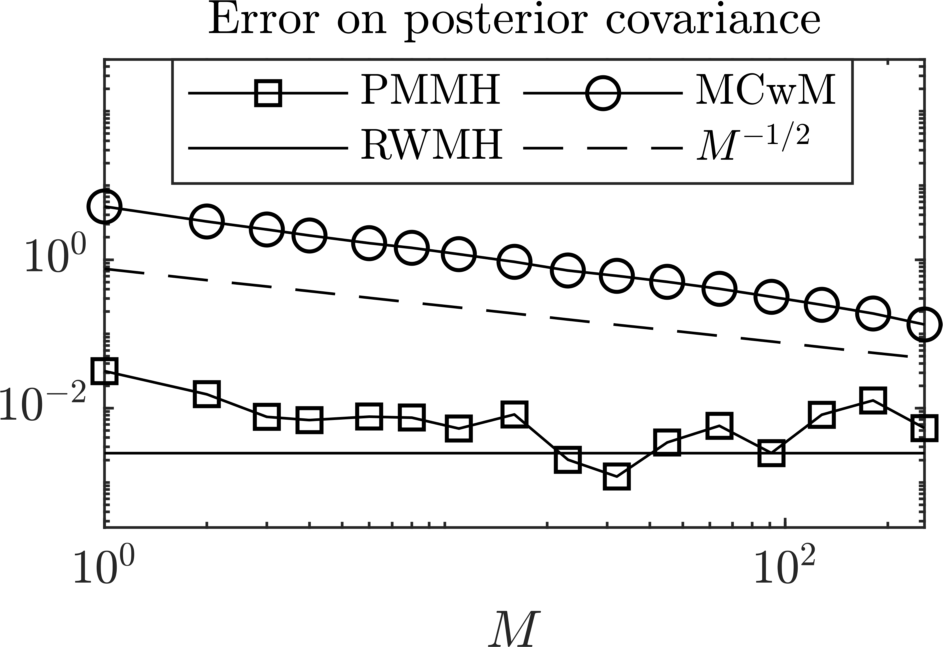} \vspace{0.3cm} \\
	\end{tabular}
	\caption{Acceptance ratio, error on posterior mean and covariance for the PMMH and the MCWM targeting the marginal posterior $\mu_{h,\mathrm{m}}$ as a function of the number of Monte Carlo samples $M$. In all plots, the horizontal dashed line corresponds to the RWMH directly targeting the same distribution. For this experiment, we fix $N = 10^6$, $h = 0.25$, and $\sigma = 0.1$.}
	\label{fig:Marg_M}
\end{figure}

\begin{figure}[h!]
	\centering
	\begin{tabular}{c@{\hskip 1cm}c}
		\multicolumn{2}{c}{\includegraphics[]{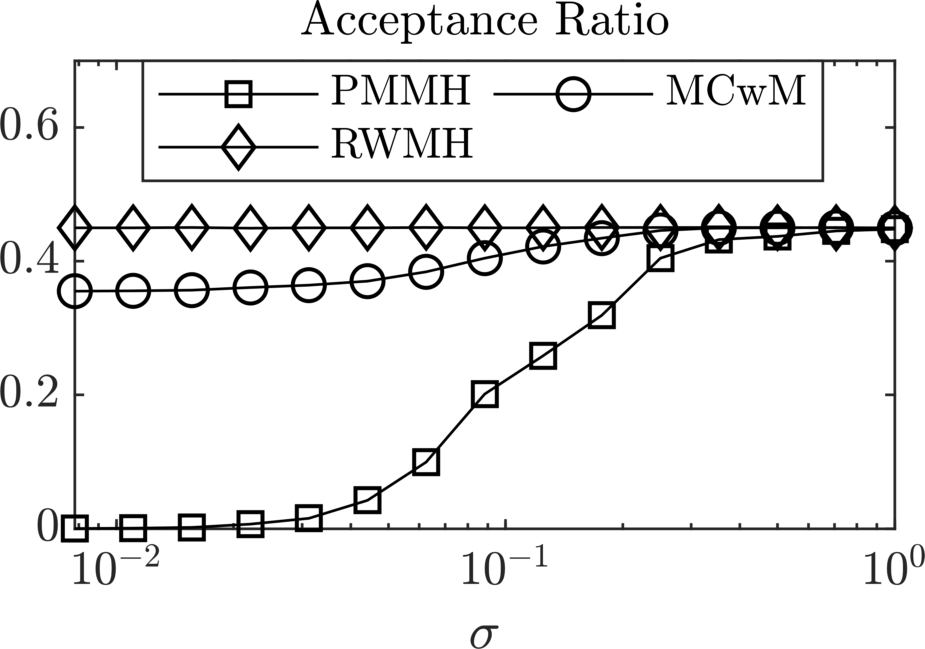}} \vspace{0.3cm} \\
		\includegraphics[]{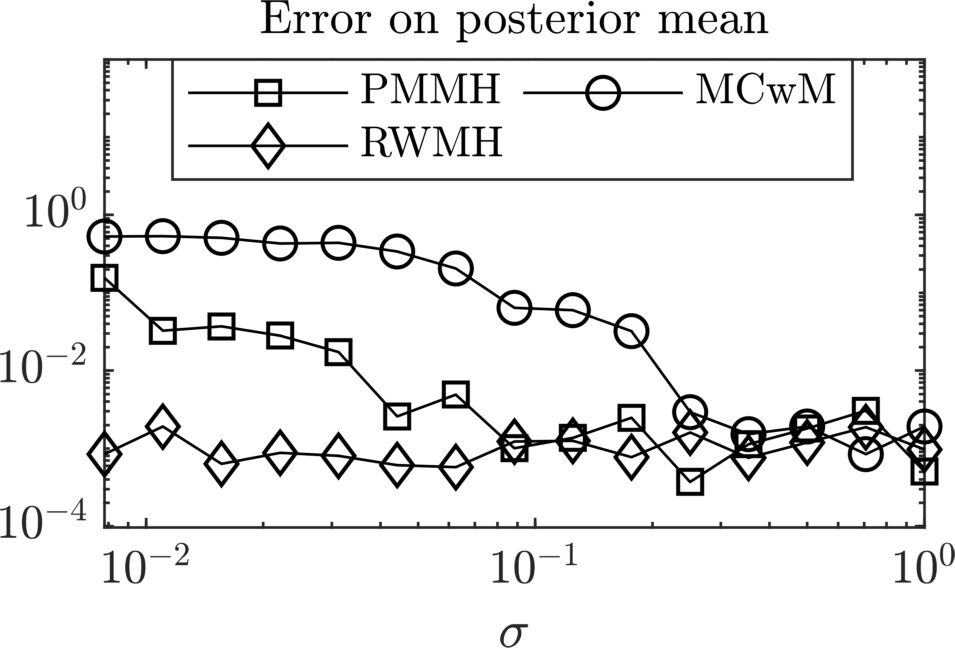} & \includegraphics[]{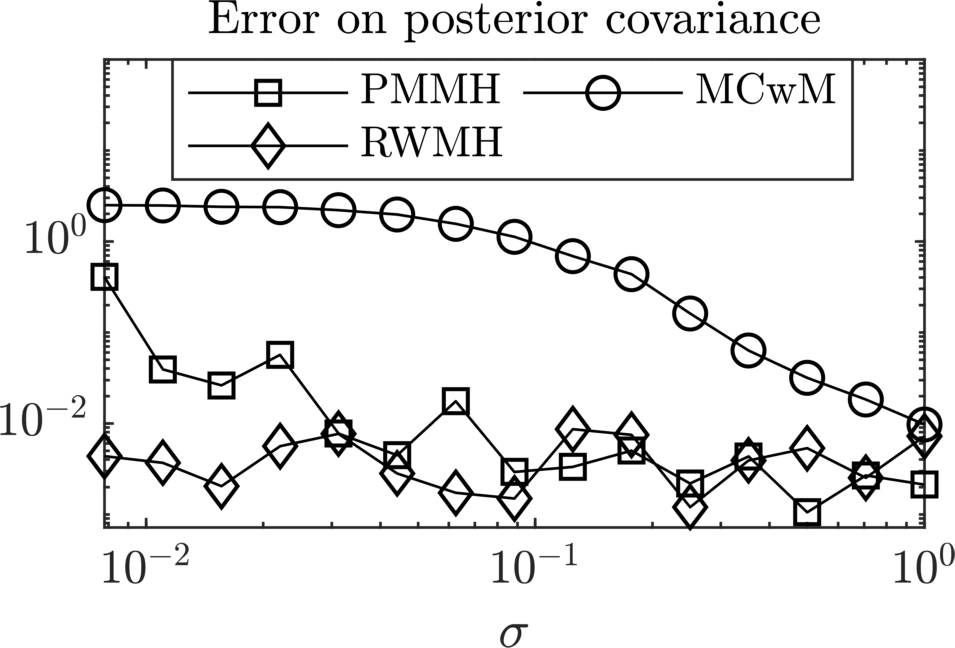} \vspace{0.3cm} \\
	\end{tabular}
	\caption{Acceptance ratio, error on posterior mean and covariance for the PMMH and the MCWM targeting the marginal posterior $\mu_{h,\mathrm{m}}$ as a function of the observation noise scale $\sigma$, compared with the RWMH targeting the same distribution. For this experiment, we fix $N = 10^6$, $h = 0.25$, and $M = 16$.}
	\label{fig:Marg_sigma}
\end{figure}

\begin{figure}[h!]
	\centering
	\begin{tabular}{c@{\hskip 1cm}c}
		\multicolumn{2}{c}{\includegraphics[]{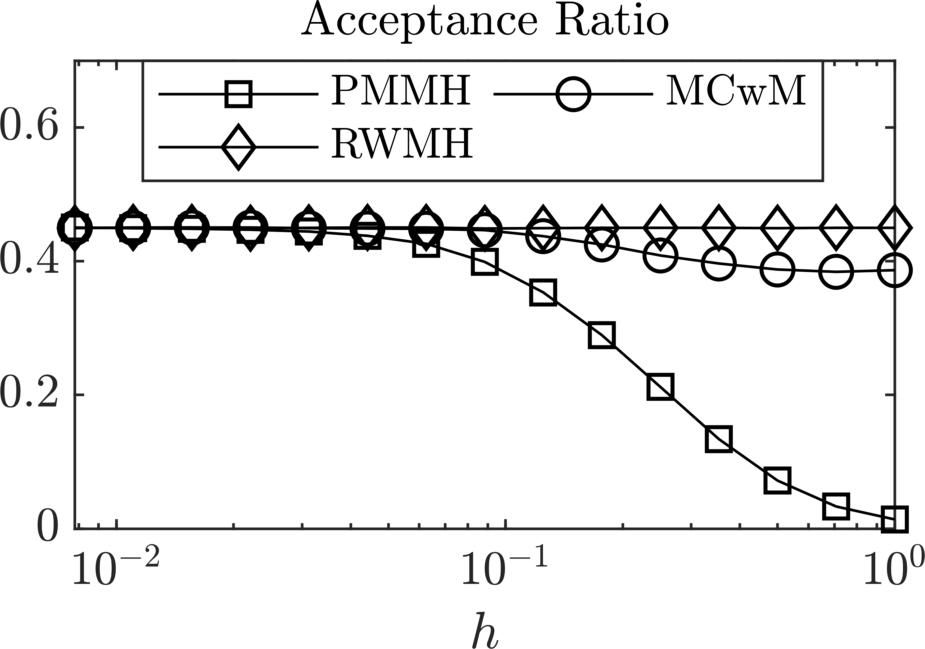}} \vspace{0.3cm} \\
		\includegraphics[]{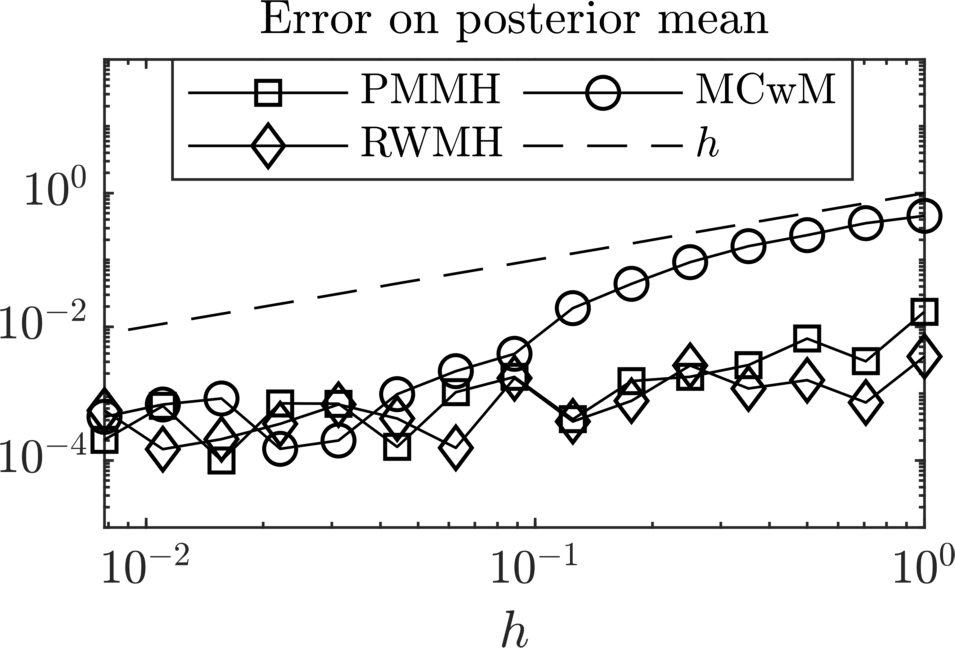} & \includegraphics[]{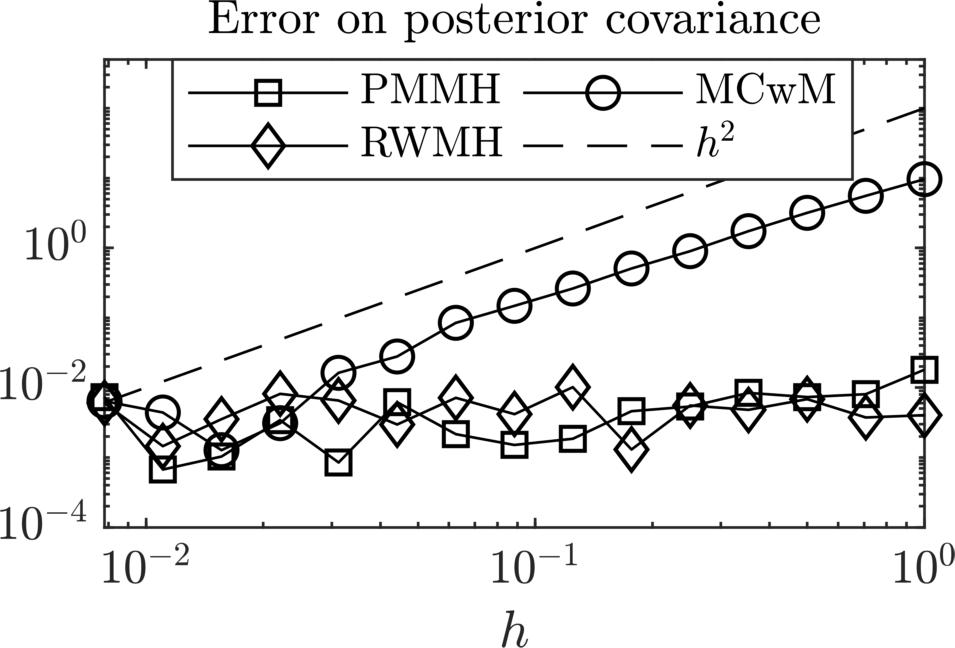} \vspace{0.3cm} \\
	\end{tabular}
	\caption{Acceptance ratio, error on posterior mean and covariance for the PMMH and the MCWM targeting the marginal posterior $\mu_{h,\mathrm{m}}$ as a function of the discretization parameter $h$, compared with the RWMH targeting the same distribution. For this experiment, we fix $N = 10^6$, $M = 16$, and $\sigma = 0.1$.}
	\label{fig:Marg_h}
\end{figure}

We consider the input and output dimensions to be equal and given by $d = m = 3$, and the matrix $A \in \R^{3\times 3}$ defining the exact forward map to have random entries distributed uniformly in $(-1,1)$. The approximated matrix $A_h$ is then given by $A_h = A + hI$, where $I$ is the $3\times 3$ identity matrix, and the randomized forward map by $\mathcal G_h(u) = A_hu + h\xi$, for $\xi \sim \mathcal N(0, I)$ and $u \in \R^3$. Observations are then obtained as $y = Au^\dagger + \beta$, where $u^\dagger = (1,2,3)^\top$ and $\beta \sim \mathcal N(0, \Gamma)$, with $\Gamma = \sigma^2I$. We are interested in comparing the performances of the PMMH and the MCwM when targeting the marginal posterior $\mu_{h,\mathrm{m}}$. Since in this case the exact marginal distribution is known by \cref{prop:LinGaussMarg}, we can compare the performances of the two algorithms with the RWMH targeting the posterior and with a Gaussian proposal $Q$ with covariance $C_Q = C_{h,\mathrm{m}}$. This same proposal mechanism is employed for the PMMH and for the MCWM. In order to have a negligible error due to the MCMC sampling, we generate for all methods chains of length $N = 10^6$. We then consider the impact on the PMMH and the MCwM of (i) the number of samples $M$ in the Monte Carlo averages to approximate the expectation of the likelihood, (ii) the observational noise scale $\sigma$ and (iii) the discretization parameter $h$. Results show that:
\begin{enumerate}
	\item The quality of the Markov chain generated by the PMMH, in terms of acceptance ratio, is degenerate if the number of samples $M$ is small, while for the MCwM the quality of the sample is robust with respect to $M$. On the other hand, the approximation of the posterior mean and covariance obtained with the PMMH is comparable -- due to the large size of the sample -- to the one of plain RWMH, while for MCwM a small sample size for the inner Monte Carlo average has a detrimental effect on the approximation of $\mu_{h,\mathrm{m}}$. We observe numerically that the posterior mean and covariance seem to converge to the marginal mean and covariance, for the MCwM, with a rate $M^{-1/2}$. Results are given in \cref{fig:Marg_M};
	\item A similar degeneracy of the MCMC sample for the PMMH as above is noticeable if the observational noise scale $\sigma$ is small, while again the MCwM does not suffer of such a negative effect. As above, though, the MCwM fails on sampling from the marginal posterior in case $\sigma$ is small. Results are given in \cref{fig:Marg_sigma};
	\item Again, the PMMH suffers from chain degeneracy in case the discretization parameter $h$ is too big. Still, if $M$ and $N$ are big enough, the approximation of the posterior is comparable to plain RWMH for any $h$, while for the MCwM results improve as $h$ gets smaller. We observe numerically that the posterior mean and covariance predicted by the MCwM converge towards the true marginal mean and covariance with rates $h$ and $h^2$, respectively. Results are given in \cref{fig:Marg_h}.
\end{enumerate}
Summarizing, it is advisable to employ the PMMH in case the randomized forward map is not too expensive to evaluate, as in this case one can increase the number of samples $M$ and obtain a reasonably-behaved sample from the exact marginal posterior. Moreover, the ratio between the discretization typical size and the observational noise scale plays a relevant role, and leads to well-behaved Markov chains in case $h/\sigma \ll 1$. On the other hand, the MCwM can be employed with a few samples in the likelihood approximation to obtain a sample from an inexact measure regardless of $h$ and $\sigma$. The targeted measure is though close to the true marginal posterior in case $h$ is small enough, or $\sigma$ is large. We refer the reader to \cite{MLR16,MRS20} for a deeper discussion on how to tune the MCwM efficiently.

\subsection{The Averaged Measure}
We now consider the averaged approximation $\mu_{h,\mathrm{a}}$ of $\mu$ given by the randomized forward map $\mathcal G_h$ defined above. In order to compute its closed-form expression, we first need a result on Gaussian random measures.

\begin{lemma}\label{lem:RandGauss} Let $\xi \sim \nu$ where $\nu=\mathcal N(m_\xi, C_\xi)$ is a $\R^m$-valued Gaussian random variable, and consider the random Gaussian measure $\mu(\xi) = \mathcal N(F(\xi), C)$ on $\R^d$, where $F \colon \R^m \to \R^d$ is an affine transformation such that $F\colon \xi \mapsto F_0 \xi + F_1$, with $F_0 \in \R^{d\times m}$ and $F_1 \in \R^m$. Then
	\begin{equation}
		\Em{\nu}{\mu(\cdot)} = \mathcal N(F(m_\xi), C + F_0 C_\xi F_0^\top). 
	\end{equation}
\end{lemma}
\begin{proof} We consider without loss of generality the case $m=d$ and $F(\xi) = \xi$, i.e., $F_0 = I$ and $F_1 = 0$, where $0$ is the zero matrix in $\R^{d\times d}$. The desired result then follows from the usual formula for affine transformations of Gaussian random variables. Let $\lambda$ be the Lebesgue measure on $\R^d$. For all $u \in \R^d$ it holds
	\begin{equation}
		\Em{\nu}{\frac{\d \mu(\cdot)}{\d \lambda}(u)} \propto \int_{\R^d} \exp\left(-\frac12 \norm{C^{-1/2}(u - \xi)}^2_2 - \frac12 \norm{C_\xi^{-1/2}(\xi - m_\xi)}_2^2\right) \dd \lambda(\xi). 
	\end{equation} 
	Completing the square inside the integral yields after algebraic simplifications and recognizing a Gaussian integral with respect to $\xi$
	\begin{equation}
		\Em{\nu}{\frac{\d \mu(\cdot)}{\d \lambda}(u)} \propto \exp\left(-\frac12\norm{(C+C_\xi)^{-1/2}(u-m_\xi)}_2^2\right).
	\end{equation}
	We then notice that Fubini's theorem gives
	\begin{equation}
		\frac{\d \Em{\nu}{\mu(\cdot)}}{\d \lambda}(u) = \Em{\nu}{\frac{\d \mu(\cdot)}{\d \lambda}(u)},
	\end{equation}
	which proves the desired result.
\end{proof}

We can now compute the averaged approximation $\mu_{h,\mathrm{a}}$.

\begin{proposition} With the notation introduced above, the average sample posterior on $\R^d$ associated to $\mathcal G_h\colon u \mapsto A_hu + h\xi$ is the Gaussian measure $\mu_{h,\mathrm{a}} = \mathcal N(m_{h,\mathrm{a}}, C_{h,\mathrm{a}})$, where
	\begin{equation}
		\begin{aligned}
			&C_{h,\mathrm{a}} = C_{h,\mathrm{s}} + h^2C_{h,\mathrm{s}}A_h^\top \Gamma^{-1}Q \Gamma^{-1} A_h C_{h,\mathrm{s}},  \\
			&m_{h,\mathrm{a}} = C_{h,\mathrm{s}}\left(A_h^\top \Gamma^{-1} y + C_0^{-1} m_0\right),
		\end{aligned}
	\end{equation}
	where $C_{h,s}$ is the covariance of the sample distribution $\mu_{h,\mathrm{s}}$ given in \eqref{eq:LinGaussSample}.
\end{proposition}
\begin{proof} The result is a direct consequence of \cref{lem:RandGauss}. Indeed, we have that the random variable $\xi$ is distributed following the measure $\nu = \mathcal N(0, Q)$ and
	\begin{equation}
	\begin{aligned}
		&m_{h,s} = F_0 \xi + F_1, \\
		&F_0 \defeq -hC_{h,s}A_h^\top \Gamma^{-1}, \quad F_1 \defeq C_{h,s}\left(A_h^\top \Gamma^{-1}y + C_0^{-1}m_0\right),
	\end{aligned}
	\end{equation}
	which proves the desired result.
\end{proof}

\begin{remark} In this linear setting, the Monte Carlo approximation $\mu_{h,\mathrm{a}}^M$ of the averaged measure is by definition a Gaussian mixture. Indeed, we have for any realization $\xi^{(i)}$ of the random variable $\xi$ in the forward model $\mathcal G_h$ that the associated posterior $\mu_{h,\mathrm{a}}^{(i)} \sim \mathcal N(m_{h,\mathrm{a}}^{(i)}, C_{h,\mathrm{a}}^{(i)})$, with
 	\begin{equation}
	\begin{aligned}
		&C_{h,\mathrm{a}}^{(i)} = C_{h,\mathrm{s}}, \\
		&m_{h,\mathrm{a}}^{(i)} = C_{h,\mathrm{s}}\left(A_h^\top \Gamma^{-1} (y - h\xi^{(i)}) + C_0^{-1} m_0\right).
	\end{aligned}
	\end{equation}
	Hence, for a random variable $u \sim \mu_{h,\mathrm{a}}^{M}$ has mean $m_{h,\mathrm{a}}^M$ and covariance $C_{h,\mathrm{a}}^M$ given by
	\begin{equation}
	\begin{aligned}
		&m_{h,\mathrm{a}}^M = \frac1M \sum_{i=1}^M m_{h,\mathrm{a}}^{(i)}, \\
		&C_{h,\mathrm{a}}^M = \frac1M \sum_{i=1}^M \left(C_{h,\mathrm{a}}^{(i)} + (m_{h,\mathrm{a}}^{(i)} - m_{h,\mathrm{a}}^M)(m_{h,\mathrm{a}}^{(i)} - m_{h,\mathrm{a}}^M)^\top\right).
	\end{aligned}
	\end{equation}
	In this case, the mean and covariance clearly do not fully characterize the distribution $\mu_{h,\mathrm{a}}^M$.
\end{remark}

\subsubsection*{Numerical Experiment}

\begin{figure}[t]
	 \centering
	\begin{tabular}{ccc}
		\includegraphics[]{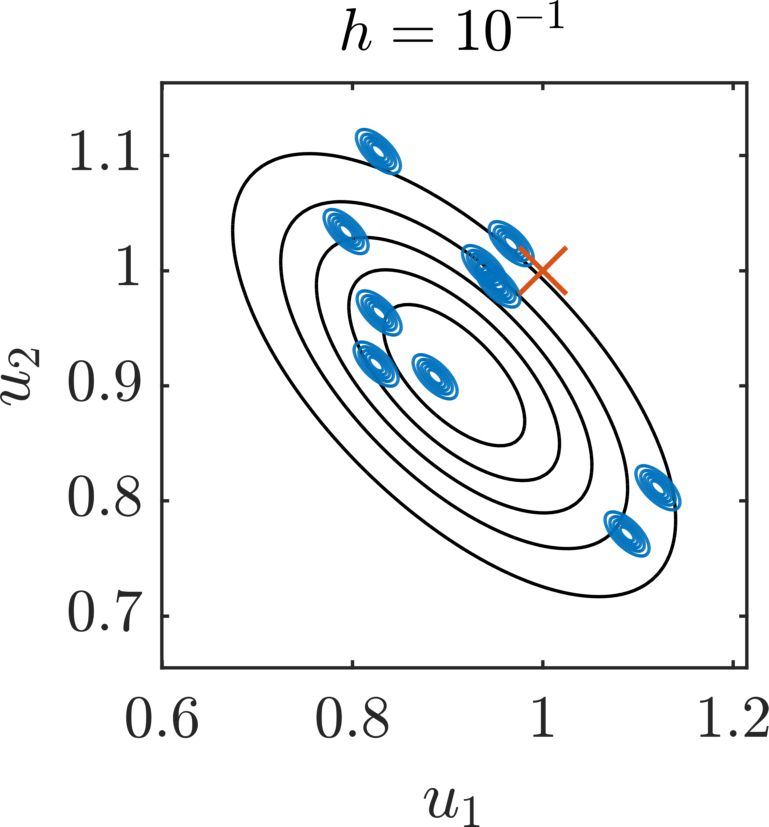} & \includegraphics[]{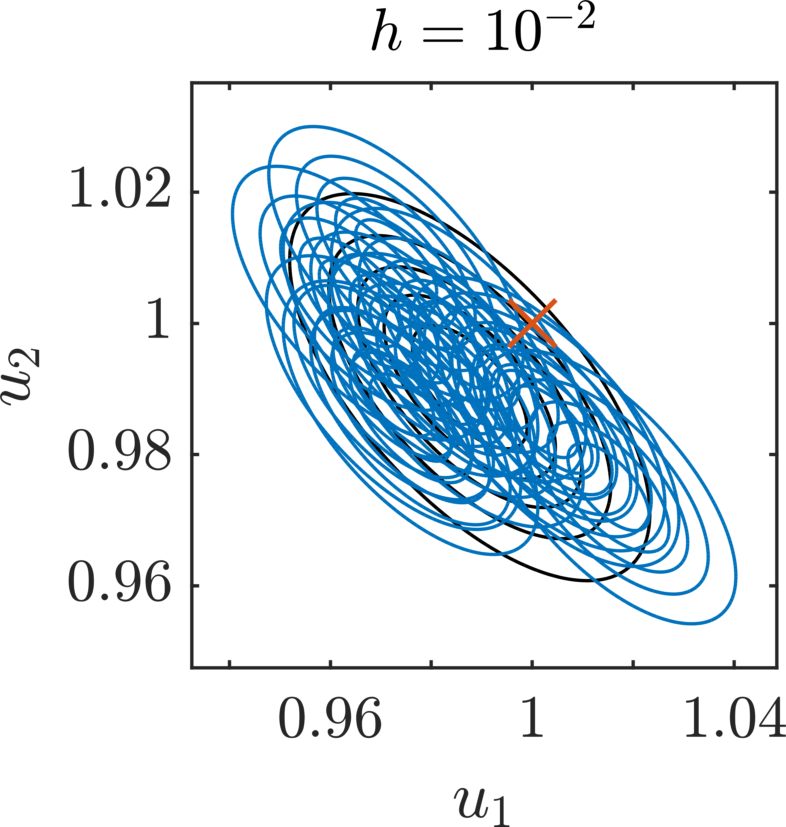} & \includegraphics[]{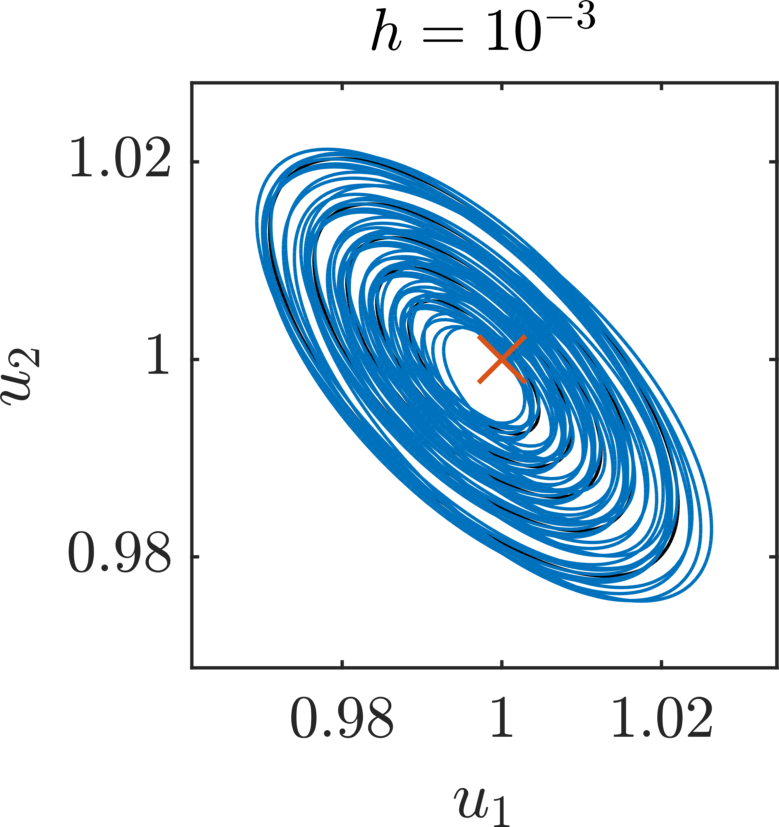}
	\end{tabular}
	\caption{Contour plot of the Monte Carlo approximation $\mu_{h,\mathrm{a}}^M$ (in blue) of the averaged measure $\mu_{h,\mathrm{a}}$ (in black) for three values of the discretization parameter $h$. The red cross denotes the true value of the parameter $u^\dagger$.}
	\label{fig:sam_M}
\end{figure}

\begin{figure}[t]
	\centering
	\begin{tabular}{cc}
		\includegraphics[]{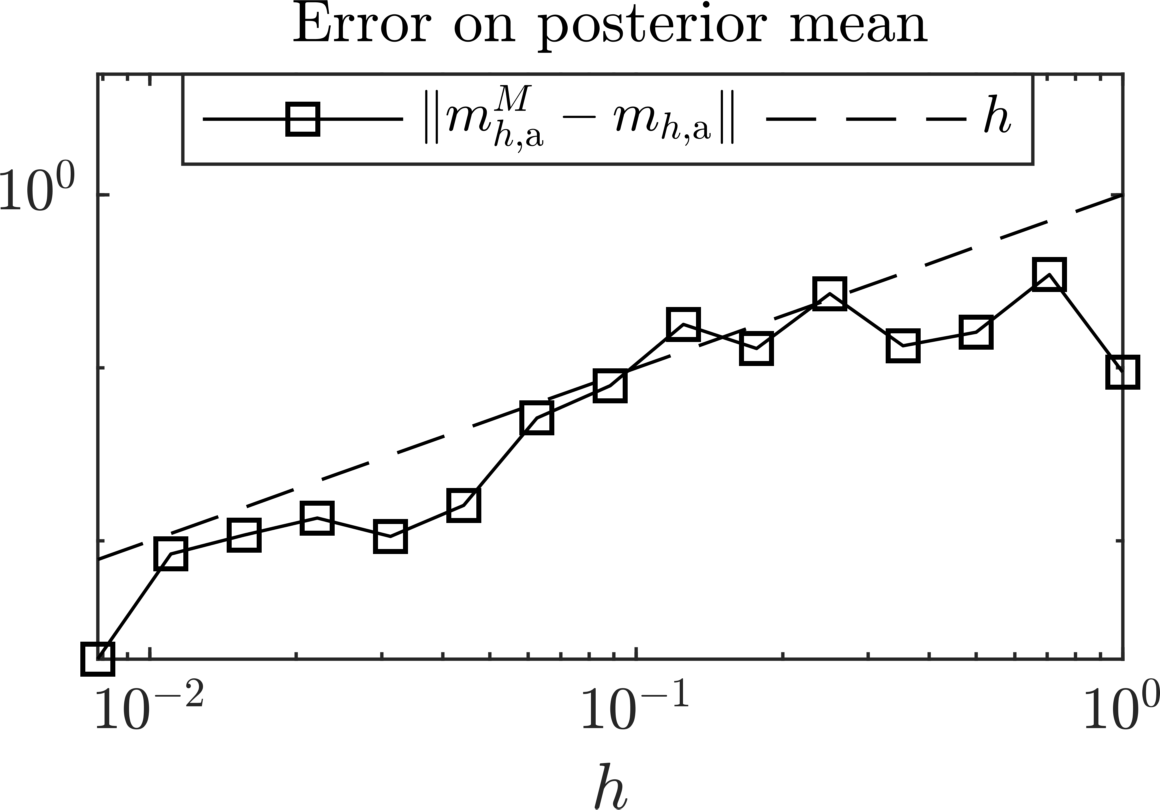} & \includegraphics[]{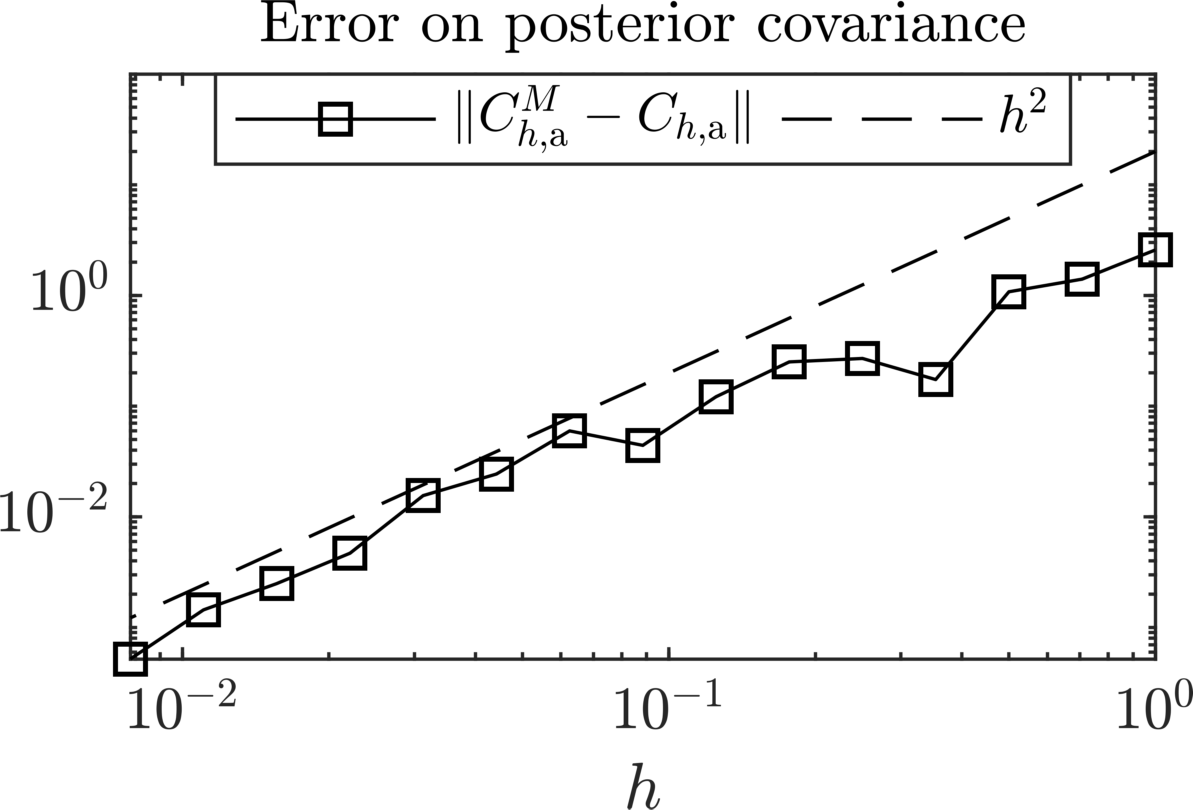}
	\end{tabular}
	\caption{Convergence of the posterior mean and covariance of the Monte Carlo approximation $\mu_{h,\mathrm{a}}^M$ of the averaged posterior $\mu_{h,\mathrm{a}}$ with respect to $h$ for a fixed value of $M = 16$.}
	\label{fig:sam_h}
\end{figure}

Unlike the marginal posterior, we argue that comparing performances of MCMC sampling methods in the context of Monte Carlo approximations of the averaged posterior would be irrelevant for this linear test case. Indeed, for any realization of the forward map, i.e., for any $\mathcal G^{(i)}(u) = A_hu + h\xi^{(i)}$, where $i = 1, \ldots, M$ and an i.i.d. sample $\{\xi^{(i)}\}$ such that $\xi^{(1)} \sim \mathcal N(0, Q)$, the RWMH can be simply employed to obtain a well-behaved sample from the corresponding posterior distribution. Tempering techniques could be employed for sampling directly from the multi-modal distribution $\mu_{h,\mathrm{a}}^M$ as well \cite{LMN21}. It is relevant, instead, to directly consider the approximation that the measure $\mu_{h,\mathrm{a}}^M$ yields of the true averaged posterior $\mu_{h,\mathrm{a}}$ for such a sample. For this purpose, we consider $d = m = 2$ and $A \in \R^{2\times 2}$ with random entries chosen uniformly in $(-1, 1)$. The approximated matrix $A_h$ is given by $A_h = A + hI$, and the randomized forward map is given by $\mathcal G_h(u) = A_hu + h\xi$, with $\xi \sim \mathcal N(0, I)$. Observations are generated as $y = Au^\dagger + \beta$, with $u^\dagger = (1,2)^\top$ and $\beta \sim \mathcal N(0, \Gamma)$, with $\Gamma = 10^{-4}I$. In \cref{fig:sam_M} we show how choosing a fixed value $M = 10$ for the Monte Carlo approximation of the averaged posterior leads to completely posterior qualities depending on the value of the discretization parameter $h$. Indeed, we see that for $h = 0.1$ a small number of samples leads to a posterior which does not ``cover'' satisfactorily the averaged posterior, while for $h = 10^{-3}$ each sample $\xi^{(i)}$ is associated to a posterior which is, for practical purposes, equal to the averaged posterior $\mu_{h,\mathrm{a}}$. The approximation of $\mu_{h,\mathrm{a}}$ by $\mu_{h,\mathrm{a}}^M$ for a fixed $M$, and a variable $h$, is highlighted again in \cref{fig:sam_M}, where we observe numerically that the convergence in the mean and the covariance is of order $h$ and $h^2$, respectively.

\section{Convergence Analysis}\label{sec:Conv}

Our goal in this section is presenting the proofs of \cref{thm:ConvMarginal,thm:ConvSample}, which are the core of our theoretical investigation. Before presenting the proofs themselves, we consider in the next \cref{sec:Likelihood} properties of Monte Carlo approximations of the likelihood function in the context of mean-square convergent randomized approximations.

\subsection{Approximation of the Likelihood}\label{sec:Likelihood}

In this section, we consider Monte Carlo approximations of the approximated random likelihood function $\exp(-\Phi_h)$ and of the random normalization constant $Z_h$. Let us remark that the properties proved in this section are related to the forward map.

We first consider the quantities involved in the marginal posterior distribution $\mu_{h,\mathrm{m}}$, i.e., the expectation of the likelihood $\exp(-\Phi_h)$ and of the normalization constant $Z_h$ with respect to the randomization of the forward model, as well as their Monte Carlo estimators. Since the Monte Carlo estimators are unbiased, we consider the approximation in terms of the variance under the measure $\nu_h$. 

\begin{proposition}\label{prop:Variance} With the notation of \cref{def:MCApprox} and if the probabilistic method $\mathcal G_h$ has mean-square order $s$ in the sense of \cref{def:MSOrder}, it holds
	\begin{alignat}{2}
		&\Vm{\nu_h}{\exp(-\Phi_h)} \leq C h^{2s}, &&\quad \mu_0\text{-a.s.},\label{eq:VarLik}\\
		&\Vm{\nu_h}{Z_h} \leq C h^{2s},  \label{eq:VarConst}
	\end{alignat}
	for a positive constant $C$ independent of $h$. Moreover, the Monte Carlo estimators satisfy
	\begin{alignat}{2}
		&\Vm{\nu_h}{\exp^M(-\Phi_h)} \leq C \frac{h^{2s}}{M}, &&\quad \mu_0\text{-a.s.},\\
		&\Vm{\nu_h}{Z_h^M} \leq C \frac{h^{2s}}{M},
	\end{alignat}
	for the same constant $C > 0$. 
\end{proposition}

\begin{proof} We start by the bound \eqref{eq:VarLik}. Since $\exp(-\Phi)$ is independent of $\nu_h$ and by definition of variance, we have
	\begin{equation}
		\Vm{\nu_h}{\exp(-\Phi_h)} = \Vm{\nu_h}{\exp(-\Phi_h) - \exp(-\Phi)} \leq \E^{\nu_h}\left[\left(\exp(-\Phi_h) - \exp(-\Phi)\right)^2\right], \quad \mu_0\text{-a.s.}
	\end{equation}
	We now remark that since $\Phi_h$ and $\Phi$ are positive $\mu_0$-a.s., then
	\begin{equation}
		\abs{\exp(-\Phi_h) - \exp(-\Phi)} \leq \abs{\Phi_h - \Phi}, \quad \mu_0\text{-a.s.},
	\end{equation}
	Hence, it holds
	\begin{equation}
		\Vm{\nu_h}{\exp(-\Phi_h)} \leq \E^{\nu_h}\left[\left(\Phi_h(u) - \Phi(u)\right)^2\right], \quad \mu_0\text{-a.s.},
	\end{equation}
	for a constant $C$ independent of $h$. The proof of \cite[Corollary 4.9]{Stu10} then yields for a positive constant $C > 0$ independent of $h$
	\begin{equation}\label{eq:ProofLikelihood_1}
		\begin{alignedat}{2}
			\Vm{\nu_h}{\exp(-\Phi_h)} &\leq C \E^{\nu_h}\left[\norm{\Gamma^{-1/2}\left(\mathcal G_h - \mathcal G\right)}_2^2\right] \\
			&\leq C \norm{\Gamma^{-1/2}}_2^2 \E^{\nu_h}\left[\norm{\mathcal G_h - \mathcal G}_2^2\right]\\
			& \leq C h^{2s}, &&\quad \mu_0\text{-a.s.},
		\end{alignedat}
	\end{equation}
	where we used in the last line that $\mathcal G_h$ has mean-square order $s$, and which proves \eqref{eq:VarLik}. We then consider the bound \eqref{eq:VarConst} for the normalizing constant and remark that Fubini's theorem, Jensen's inequality and \eqref{eq:ProofLikelihood_1} yield
	\begin{equation}
	\begin{aligned}
		\Vm{\nu_h}{Z_h} &= \Vm{\nu_h}{\Em{\mu_0}{\exp(-\Phi_h)}} \leq \Em{\mu_0}{\Vm{\nu_h}{\exp(-\Phi_h)}} \leq Ch^{2s},
	\end{aligned}
	\end{equation}
	which proves \eqref{eq:VarConst}. For the Monte Carlo estimators, the desired result holds since the samples are i.i.d.
\end{proof}
For the Monte Carlo approximation $\mu_{h,\mathrm{a}}^M$ of the averaged posterior $\mu_{h,\mathrm{a}}$, the application of usual techniques for assessing the quality of Monte Carlo estimators is not as direct. In particular, we rewrite
\begin{equation}
	\frac{\d \mu_{h,\mathrm{a}}^M}{\d \mu_0}(u) = \frac{\expapp^M\left(-\Phi_h(u)\right)}{Z_h^M},
\end{equation}
where
\begin{equation}
	\expapp^M\left(-\Phi_h(u)\right) \defeq \frac1M \sum_{i=1}^M\exp\left(-\Phi_h^{(i)}(u)\right)\frac{Z_h^M}{Z_h^{(i)}},
\end{equation}
i.e., $\expapp^M$ is the Monte Carlo estimator of the likelihood function weighted by the normalization constants. We remark that since by construction it holds $\mu_{h,\mathrm{a}}^M(X) = 1$, then
\begin{equation}
	Z_h^M = \int_X \expapp^M\left(-\Phi_h(u)\right) \dd \mu_0(u),
\end{equation}
and that thus $\expapp^M(-\Phi_h)$ and $\exp^M(-\Phi_h)$ have the same expectation with respect to the prior $\mu_0$, i.e.,
\begin{equation}
	\Em{\mu_0}{\expapp^M\left(-\Phi_h\right) - \exp^M\left(-\Phi_h\right)} = 0.
\end{equation}
Nevertheless, the quantity $\expapp^M(-\Phi_h)$ is not an unbiased estimator of $\exp(-\Phi_h)$ due to the weighting of the average. We can still consider the quantity
\begin{equation}
	\expapp(-\Phi_h) \defeq \exp\left(-\Phi_h\right) \frac{\Em{\nu_h}{Z_h}}{Z_h},
\end{equation}
and study the approximation of $\expapp(-\Phi_h)$ by $\expapp^M(-\Phi_h)$. We remark that $\expapp^M(-\Phi_h)$ is biased in approximating $\expapp(-\Phi_h)$, as well, and we therefore study the approximation in terms of the mean-square error, as shown in the following result. 

\begin{proposition}\label{prop:MSEExpApp} Under \cref{as:Marginal}, it holds
	\begin{equation}
		\Em{\mu_0}{\Em{\nu_h}{\left(\expapp^M\left(-\Phi_h\right) - \Em{\nu_h}{\expapp(-\Phi)}\right)^2}} \leq C\frac{h^{2s}}{M},
	\end{equation}
	for a positive constant $C$ independent of $h$ and $M$.
\end{proposition}

In order to prove \cref{prop:MSEExpApp}, we first introduce an additional result which allows to bound the variance of the Radon--Nykodim derivative of $\mu_{h,\mathrm{s}}$ under the randomization of the forward model.

\begin{lemma}\label{lem:Variance_2} Under \cref{as:Marginal}, it holds
	\begin{equation}
		\Em{\mu_0}{\Vm{\nu_h}{\frac{\exp(-\Phi_h)}{Z_h}}} \leq C h^{2s},
	\end{equation}
	for a positive constant $C$ independent of $h$.
\end{lemma}
\begin{proof} Since $\exp(-\Phi)$ and $Z$ are independent of $\nu_h$ it holds
	\begin{equation}
		\Vm{\nu_h}{\frac{\exp(-\Phi_h)}{Z_h}} = \Vm{\nu_h}{\frac{\exp(-\Phi_h)}{Z_h} - \frac{\exp(-\Phi)}{Z}}, \quad \mu_0\text{-a.s.},
	\end{equation}
	which, applying $\V(X + Y) \leq 2\V(X) + 2\V(Y)$, yields
	\begin{equation}\label{eq:ProofLikelihood_2}
		\Vm{\nu_h}{\frac{\exp(-\Phi_h)}{Z_h}} \leq \frac{2}{Z} \Vm{\nu_h}{\exp(-\Phi_h)} + 2 \Vm{\nu_h}{\exp(-\Phi_h)\left(\frac1{Z_h} - \frac1Z\right)}, \quad \mu_0\text{-a.s.}
	\end{equation}
	We now remark that for any positive $a$ and $b$ it holds
	\begin{equation}
		\left(\frac1a - \frac1b\right)^2 \leq \max\{a^{-4}, b^{-4}\}\abs{a-b}^2,
	\end{equation}
	and therefore Hölder's inequality implies
	\begin{equation}\label{eq:ProofLikelihood_3}
		\begin{aligned}
			\Vm{\nu_h}{\exp(-\Phi_h)\left(\frac1{Z_h} - \frac1Z\right)} 
			&\leq \Em{\nu_h}{\exp(-2\Phi_h)\left(\frac1{Z_h} - \frac1Z\right)^2}\\
			&\leq
			\begin{alignedat}[t]{2}
				&\norm{\exp(-2\Phi_h)\max\{Z_h^{-4}, Z^{-4}\}}_{L^\infty_{\nu_h}(\Omega)} \\
				&\quad \times \Em{\nu_h}{(Z_h - Z)^2},&&\quad \mu_0\text{-a.s.}.
			\end{alignedat} \\
		\end{aligned}
	\end{equation}
	Since $\Phi_h > 0$ $\mu_0$- and $\nu_h$-a.s. and applying $\norm{\max\{f,g\}}_{L^\infty_{\nu_h}} \leq \max\{\norm{f}_{L^\infty_{\nu_h}},\norm{g}_{L^\infty_{\nu_h}}\}$, we now have by \cref{as:Marginal}
	\begin{equation}
		\norm{\exp(-2\Phi_h)\max\{Z_h^{-4}, Z^{-4}\}}_{L^\infty_{\nu_h}(\Omega)} \leq \max\left\{\norm{Z_h^{-4}}_{L^\infty_{\nu_h}(\Omega)}, Z^{-4}\right\} \leq C, \qquad \mu_0\text{-a.s.},
	\end{equation}
	where $C = \max\{C_3^4,C_4^4\}$. Moreover, Jensen's inequality and Fubini's theorem give
	\begin{equation}
		\Em{\nu_h}{(Z_h - Z)^2} = \Em{\nu_h}{\Em{\mu_0}{\exp(-\Phi_h) - \exp(-\Phi)}^2} \leq \Em{\mu_0}{\Em{\nu_h}{\left(\exp(-\Phi_h) - \exp(-\Phi)\right)^2}},
	\end{equation} 
	which, proceeding as in the proof of \cref{prop:Variance}, gives
	\begin{equation}
		\Em{\nu_h}{(Z_h - Z)^2} \leq Ch^{2s}.
	\end{equation}
	Together with \cref{prop:Variance}, \eqref{eq:ProofLikelihood_2}, \eqref{eq:ProofLikelihood_3} and under \cref{as:Marginal}, this finally implies 
	\begin{equation}
		\Em{\mu_0}{\Vm{\nu_h}{\frac{\exp(-\Phi_h)}{Z_h}}} \leq Ch^{2s},
	\end{equation}
	which is the desired result.
\end{proof}

It is now possible to prove \cref{prop:MSEExpApp}.

\begin{proof}[Proof of \cref{prop:MSEExpApp}] The classic bias-variance decomposition of the mean-square error gives
	\begin{equation}\label{eq:BV_BiasVariance}
		\Em{\nu_h}{\left(\expapp^M\left(-\Phi_h\right) - \Em{\nu_h}{\expapp(-\Phi)}\right)^2} \leq \bias^{\nu_h}\left(\expapp^M\left(-\Phi_h\right)\right)^2 + \Vm{\nu_h}{\expapp^M\left(-\Phi_h\right)},
	\end{equation}
	$\mu_0$-a.s., where
	\begin{equation}
		\bias^{\nu_h}\left(\expapp^M\left(-\Phi_h\right)\right) \defeq \Em{\nu_h}{\expapp^M\left(-\Phi_h\right) - \expapp\left(-\Phi_h\right)}.
	\end{equation}
	We first consider the bias term. Replacing the definition of $\expapp(-\Phi_h)$ and $\expapp^M(-\Phi_h)$, we notice that it holds
	\begin{equation}
		\bias^{\nu_h}\left(\expapp^M\left(-\Phi_h\right)\right) = \frac1M \sum_{i=1}^M \Em{\nu_h}{\frac{\exp\left(-\Phi_h^{(i)}\right)}{Z_h^{(i)}}\left(Z_h^M - \Em{\nu_h}{Z_h}\right)}.
	\end{equation}
	Indeed, since $\exp(-\Phi_h^{(i)}) / Z_h^{(i)} \sim \exp(-\Phi_h) / Z_h$, we have
	\begin{equation}
		\frac1M\sum_{i=1}^M \Em{\nu_h}{Z_h} \Em{\nu_h}{\frac{\exp\left(-\Phi_h^{(i)}\right)}{Z_h^{(i)}} - \frac{\exp\left(-\Phi_h\right)}{Z_h}} = 0.
	\end{equation}
	By linearity of the expectation, the Cauchy--Schwarz inequality and \cref{prop:Variance} we then obtain
	\begin{equation}
		\begin{aligned}
			\bias^{\nu_h}\left(\expapp^M\left(-\Phi_h\right)\right)^2 &= \Em{\nu_h}{\left(Z_h^M - \Em{\nu_h}{Z_h}\right)\frac1M \sum_{i=1}^M \frac{\exp\left(-\Phi_h^{(i)}\right)}{Z_h^{(i)}}} \\
			&\leq \Vm{\nu_h}{Z_h^M}\Em{\nu_h}{\left(\frac1M \sum_{i=1}^M 			\frac{\exp\left(-\Phi_h^{(i)}\right)}{Z_h^{(i)}}\right)^2}.\\
			&\leq C\frac{h^{2s}}{M} \Em{\nu_h}{\left(\frac1M \sum_{i=1}^M 	\frac{\exp\left(-\Phi_h^{(i)}\right)}{Z_h^{(i)}}\right)^2}, \qquad \mu_0\text{-a.s.},
		\end{aligned}
	\end{equation}
	for a constant $C > 0$. We then remark that since $\exp(-\Phi_h^{(i)}) / Z_h^{(i)} \iid \exp(-\Phi_h) / Z_h$ and by definition of variance
	\begin{equation}
		\begin{aligned}
			\Em{\nu_h}{\left(\frac1M \sum_{i=1}^M \frac{\exp\left(-\Phi_h^{(i)}\right)}{Z_h^{(i)}}\right)^2} 
			&= \Vm{\nu_h}{\frac1M \sum_{i=1}^M \frac{\exp\left(-\Phi_h^{(i)}\right)}{Z_h^{(i)}}} + \Em{\nu_h}{\frac1M \sum_{i=1}^M \frac{\exp\left(-\Phi_h^{(i)}\right)}{Z_h^{(i)}}}^2\\
			&= \frac1M \Vm{\nu_h}{\frac{\exp\left(-\Phi_h\right)}{Z_h}} + \Em{\nu_h}{ \frac{\exp\left(-\Phi_h\right)}{Z_h}}^2.
		\end{aligned}
	\end{equation}
	Hence, there exists a constant $C > 0$ such that
	\begin{equation}\label{eq:BV_Bias}
		\bias^{\nu_h}\left(\expapp^M\left(-\Phi_h\right)\right)^2 \leq C\left(\frac1M \Vm{\nu_h}{\frac{\exp\left(-\Phi_h\right)}{Z_h}} + \Em{\nu_h}{ \frac{\exp\left(-\Phi_h\right)}{Z_h}}^2\right)\frac{h^{2s}}{M}, \quad \mu_0\text{-a.s.}
	\end{equation}
	We now consider the variance term in \eqref{eq:BV_BiasVariance}. For two $\mu$-distributed random variables $X$ and $Y$, it is possible to show that
	\begin{equation}
		\Vm{\mu}{XY} \leq 2 \Vm{\mu}{X}\norm{Y}_{L^\infty_\mu} + 2 \Em{\mu}{X}^2 \Vm{\mu}{Y}.
	\end{equation}
	Applying this inequality to $\expapp^M(-\Phi_h)$ yields for a constant $C > 0$ independent of $h$ and $M$
	\begin{equation}\label{eq:BV_Variance}
		\begin{aligned}
			\Vm{\nu_h}{\expapp^M(-\Phi_h)} &\leq 
			\begin{aligned}[t]
				&2\Vm{\nu_h}{\frac1M \sum_{i=1}^M \frac{\exp\left(-\Phi_h^{(i)}\right)}{Z_h^{(i)}}}\norm{Z_h^M}_{L_{\nu_h}^\infty(\Omega)} \\
				&\quad + 2 \Em{\nu_h}{\frac1M \sum_{i=1}^M \frac{\exp\left(-\Phi_h^{(i)}\right)}{Z_h^{(i)}}}^2\Vm{\nu_h}{Z_h^M}
			\end{aligned}\\
			&=
			\begin{aligned}[t]
			 &\frac{2}{M} \Vm{\nu_h}{\frac{\exp\left(-\Phi_h\right)}{Z_h}}\norm{Z_h^M}_{L_{\nu_h}^\infty(\Omega)} \\
			 &\quad + C \Em{\nu_h}{\frac{\exp\left(-\Phi_h\right)}{Z_h}}^2 \frac{h^{2s}}{M}, \quad \mu_0\text{-a.s.},
			\end{aligned}
		\end{aligned}
	\end{equation}
	where we used that $\exp(-\Phi_h^{(i)}) / Z_h^{(i)} \iid \exp(-\Phi_h) / Z_h$ and \cref{prop:Variance} for the second term. Therefore, in light of \eqref{eq:BV_BiasVariance}, \eqref{eq:BV_Bias}, and \eqref{eq:BV_Variance}, there exists a constant $C > 0$ independent of $h$ and $M$ and such that
	\begin{equation}\label{eq:BV_BeforeMu0}
		\Em{\nu_h}{\left(\expapp^M\left(-\Phi_h\right) - \Em{\nu_h}{\expapp(-\Phi)}\right)^2} \leq C(I_1 + I_2), \quad \mu_0\text{-a.s.},
	\end{equation}
	where
	\begin{equation}
		\begin{aligned}
			I_1 &\defeq \frac1M\left(\norm{Z_h^M}_{L_{\nu_h}^\infty(\Omega)} + \frac{h^{2s}}{M}\right)\Vm{\nu_h}{\frac{\exp\left(-\Phi_h\right)}{Z_h}}, \\
			I_2 &\defeq \Em{\nu_h}{\frac{\exp\left(-\Phi_h\right)}{Z_h}}^2 \frac{h^{2s}}{M}.
		\end{aligned}
	\end{equation}
	We now take the expectation $\E^{\mu_0}$ with respect to the prior on both sides of \eqref{eq:BV_BeforeMu0}. For $I_1$, it holds under \cref{as:Marginal} and due to \cref{lem:Variance_2}
	\begin{equation}\label{eq:BV_Mu0I1}
		\Em{\mu_0}{I_1} \leq C\frac{h^{2s}}{M}.
	\end{equation}
	For $I_2$, since $\exp(-\Phi_h) \leq 1$ $\nu_h$-a.s. it holds
	\begin{equation}\label{eq:BV_Mu0I2}
		\Em{\mu_0}{I_2} \leq \Em{\nu_h}{Z_h^{-1}}^2\frac{h^{2s}}{M} \leq C \frac{h^{2s}}{M},
	\end{equation}
	where the second inequality holds under \cref{as:Marginal} for a constant $C > 0$ independent of $h$ and $M$. Combining \eqref{eq:BV_BeforeMu0}, \eqref{eq:BV_Mu0I1}, and \eqref{eq:BV_Mu0I2} then yields the desired result.
\end{proof}

\subsection{Approximation of the Marginal Posterior} 

In this section we prove \cref{thm:ConvMarginal}, i.e., the convergence of the Monte Carlo estimator of the marginal approximate posterior $\mu_{h,\mathrm{m}}$ with respect to the discretization parameter $h$ and the number of samples $M$. Let us remark that the proof is similar to the proof of \cite[Theorem 3.1]{LST18}.

\begin{proof}[Proof of \cref{thm:ConvMarginal}] We first notice that it holds
	\begin{equation}
		\sqrt{\frac{\d \mu_{h,\mathrm{m}}}{\d \mu_0}} - \sqrt{\frac{\d \mu_{h,\mathrm{m}}^M}{\d \mu_0}} = 
		\begin{aligned}[t]
			&\frac{\sqrt{\Em{\nu_h}{\exp\left(-\Phi_h\right)}} - \sqrt{\exp^M\left(-\Phi_h\right)}}{\sqrt{\Em{\nu_h}{Z_h}}} \\
			&\quad + \sqrt{\exp^M\left(-\Phi_h\right)}\left(\frac1{\sqrt{\Em{\nu_h}{Z_h}}} - \frac1{\sqrt{Z_h^M}} \right).
		\end{aligned}
	\end{equation}
	Therefore, by definition of the Hellinger distance and applying the inequality $(a+b)^2 \leq 2a^2 + 2b^2$ we obtain
	\begin{equation}
		\Hell\left(\mu_{h,\mathrm{m}}, \mu_{h,\mathrm{m}}^M\right)^2 \leq I_1 + I_2, \quad \nu_h\text{-a.s.},
	\end{equation}
	where $I_1$ and $I_2$ are the positive quantities defined as
	\begin{equation}\label{eq:ProofMarg_I1I2}
		\begin{aligned}
			I_1 &\defeq \frac1{\Em{\nu_h}{Z_h}} \E^{\mu_0}\left[ \left(\sqrt{\Em{\nu_h}{\exp\left(-\Phi_h\right)}} - \sqrt{\exp^M\left(-\Phi_h\right)}\right)^2\right], \\
			I_2 &\defeq Z_h^M\left(\frac1{\sqrt{\Em{\nu_h}{Z_h}}} - \frac1{\sqrt{Z_h^M}} \right)^2 .
		\end{aligned}
	\end{equation}
	We first consider $I_1$. Applying the inequality
	\begin{equation}
		(a-b)^2 \leq \frac{(a^2-b^2)^2}{a^2+b^2},
	\end{equation}
	valid for $a,b > 0$, we obtain
	\begin{equation}
		\Em{\nu_h}{Z_h} I_1 \leq \E^{\mu_0}\left[\frac{\left(\Em{\nu_h}{\exp\left(-\Phi_h\right)} - \exp^M\left(-\Phi_h\right)\right)^2}{\Em{\nu_h}{\exp\left(-\Phi_h\right)} + \exp^M\left(-\Phi_h\right)}\right], \quad \nu_h\text{-a.s.},
	\end{equation}
	and Hölder's inequality then yields
	\begin{equation}
		\Em{\nu_h}{Z_h}I_1 \leq
		\begin{aligned}[t]
			&\Em{\mu_0}{\left(\Em{\nu_h}{\exp\left(-\Phi_h\right)} - \exp^M\left(-\Phi_h\right)\right)^2} \\
			&\quad \times \norm{\left(\Em{\nu_h}{\exp\left(-\Phi_h\right)} + \exp^M\left(-\Phi_h\right)\right)^{-1}}_{L^\infty_{\mu_0}(X)}, \quad \nu_h\text{-a.s.}
		\end{aligned}
	\end{equation}
	We now take the expectation $\E^{\nu_h}$ on both sides and Hölder's inequality to obtain
	\begin{equation}
		\Em{\nu_h}{Z_h}\Em{\nu_h}{I_1} \leq
	\begin{aligned}[t]
			 &\Em{\nu_h}{\Em{\mu_0}{\left(\Em{\nu_h}{\exp\left(-\Phi_h\right)} - \exp^M\left(-\Phi_h\right)\right)^2}} \\
			&\quad \times \norm{\norm{\left(\Em{\nu_h}{\exp\left(-\Phi_h\right)} + \exp^M\left(-\Phi_h\right)\right)^{-1}}_{L^\infty_{\mu_0}(X)}}_{L^\infty_{\nu_h}(\Omega)}.
	\end{aligned}
	\end{equation}
	For the second factor in the right-hand side, we remark that for $a,b > 0$ it holds $(a+b)^{-1} \leq \min\{a^{-1},b^{-1}\}$, and since $\norm{\min\{f,g\}}_{L^\infty_{\mu_0}} \leq \min\{\norm{f}_{L^\infty_{\mu_0}},\norm{g}_{L^\infty_{\mu_0}}\}$, we get
	\begin{equation}
	\begin{aligned}
		&\norm{\left(\Em{\nu_h}{\exp\left(-\Phi_h\right)} + \exp^M\left(-\Phi_h\right)\right)^{-1}}_{L^\infty_{\mu_0}(X)}  \\
		&
		\begin{aligned}[t]
		\qquad \qquad \qquad &\leq \norm{\min\left\{\Em{\nu_h}{\exp\left(-\Phi_h\right)}^{-1}, \exp^M\left(-\Phi_h\right)^{-1}\right\}}_{L^\infty_{\mu_0}(X)}\\
		&\leq \min\left\{\norm{\Em{\nu_h}{\exp\left(-\Phi_h\right)}^{-1}}_{L^\infty_{\mu_0}(X)}, \norm{\exp^M\left(-\Phi_h\right)^{-1}}_{L^\infty_{\mu_0}(X)}\right\}\\
		&\leq 
		\min\left\{\Em{\nu_h}{\norm{\exp\left(\Phi_h\right)}_{L^\infty_{\mu_0}(X)}}, \norm{\exp^M\left(\Phi_h\right)}_{L^\infty_{\mu_0}(X)}\right\} \leq C, \qquad \nu_h\text{-a.s.},
		\end{aligned}
	\end{aligned}
	\end{equation}
	where we applied Jensen's inequality and the discrete Jensen inequality to in the last line, and where $C = C_1$ is given in \cref{as:Marginal}. Applying Fubini's theorem, under \cref{as:Marginal} and noticing that $\Em{\nu_h}{\exp(-\Phi_h)} = \Em{\nu_h}{\exp^M(-\Phi_h)}$ we then obtain
	\begin{equation}
		\Em{\nu_h}{I_1} \leq C \Em{\mu_0}{\Vm{\nu_h}{\exp^M\left(-\Phi_h\right)}},
	\end{equation}
	which implies by \cref{prop:Variance}
	\begin{equation}\label{eq:ProofMarg_I1}
		\Em{\nu_h}{I_1} \leq C\frac{h^{2s}}{M}.
	\end{equation}
	We now consider $I_2$ in \eqref{eq:ProofMarg_I1I2}. Using the inequality
	\begin{equation}
		\left(\frac1{\sqrt{a}} - \frac1{\sqrt{b}}\right)^2 \leq \frac14 \max\{a^{-3},b^{-3}\}(a-b)^2,
	\end{equation}
	valid for $a,b > 0$, we obtain
	\begin{equation}
		I_2 \leq \frac{Z_h^M}4  \max\{\Em{\nu_h}{Z_h}^{-3}, (Z_h^M)^{-3}\}\left(\Em{\nu_h}{Z_h} - Z_h^M\right)^2.
	\end{equation}
	Taking the expectation $\E^{\nu_h}$ on both sides, replacing $\Em{\nu_h}{Z_h} = \Em{\nu_h}{Z_h^M}$, and applying Hölder's inequality yields
	\begin{equation}
		\Em{\nu_h}{I_2} \leq \frac14 \norm{Z_h^M}_{L^\infty_{\nu_h}(\Omega)} \norm{\max\{\Em{\nu_h}{Z_h}^{-3}, (Z_h^M)^{-3}\}}_{L^\infty_{\nu_h}(\Omega)}\Vm{\nu_h}{Z_h^M}.
	\end{equation}
	We now have $\norm{\max\{f,g\}}_{L^\infty_{\nu_h}} \leq \max\{\norm{f}_{L^\infty_{\nu_h}}, \norm{g}_{L^\infty_{\nu_h}}\}$ and therefore
	\begin{equation}
		\Em{\nu_h}{I_2} \leq \frac14 \norm{Z_h^M}_{L^\infty_{\nu_h}(\Omega)} \max\left\{\Em{\nu_h}{Z_h}^{-3}, \norm{(Z_h^M)^{-3}}_{L^\infty_{\nu_h}(\Omega)}\right\}\Vm{\nu_h}{Z_h^M}.
	\end{equation}
	Hence, under \cref{as:Marginal} it holds
	\begin{equation}
		\Em{\nu_h}{I_2} \leq C \Vm{\nu_h}{Z_h^M},
	\end{equation}
	where $C$ is a positive constant independent of $h$ and $M$. An application of \cref{prop:Variance} then implies
	\begin{equation}
		\Em{\nu_h}{I_2} \leq \frac{h^{2s}}{M},
	\end{equation}
	which concludes the proof with \eqref{eq:ProofMarg_I1I2} and \eqref{eq:ProofMarg_I1}.
\end{proof}

\subsection{Approximation of the Averaged Posterior} 

In this section we prove \cref{thm:ConvSample}, the main result of convergence for the Monte Carlo estimator $\mu_{h,\mathrm{a}}^M$ of the averaged posterior distribution $\mu_{h,\mathrm{a}}$.

\begin{proof}[Proof of \cref{thm:ConvSample}] We proceed similarly to \cref{thm:ConvMarginal} and write
	\begin{equation}
	\begin{aligned}
		\sqrt{\frac{\d \mu_{h,\mathrm{a}}}{\d \mu_0}} - \sqrt{\frac{\d \mu_{h,\mathrm{a}}^M}{\d \mu_0}} &= 
		\begin{aligned}[t]
			&\sqrt{\Em{\nu_h}{\frac{\exp\left(-\Phi_h\right)}{Z_h}}} - \frac{\sqrt{\expapp^M(-\Phi_h)}}{\sqrt{\Em{\nu_h}{Z_h}}}\\
			&+ \sqrt{\expapp^M(-\Phi_h)}\left(\frac1{\sqrt{\Em{\nu_h}{Z_h}}} - \frac1{Z_h^M}\right)
		\end{aligned} \\
		&=
		\begin{aligned}[t]
			&\frac1{\sqrt{\Em{\nu_h}{Z_h}}}\left(\sqrt{\Em{\nu_h}{\expapp(-\Phi_h)}} - \sqrt{\expapp^M(-\Phi_h)}\right)\\
			&+ \sqrt{\expapp^M(-\Phi_h)}\left(\frac1{\sqrt{\Em{\nu_h}{Z_h}}} - \frac1{Z_h^M}\right),
		\end{aligned}
	\end{aligned}
	\end{equation}
	which, applying $(a+b)^2 \leq 2a^2 + 2b^2$, yields
	\begin{equation}\label{eq:ProofSam_I1I2}
		\Hell\left(\mu_{h,\mathrm{a}}, \mu_{h,\mathrm{a}}^M\right) \leq I_1 + I_2, \quad \nu_h\text{-a.s.},
	\end{equation}
	and where, since $\E^{\mu_0}[\expapp^M(-\Phi_h)] = Z_h^M$, it holds
	\begin{equation}
	\begin{aligned}
		I_1 &\defeq \frac1{\Em{\nu_h}{Z_h}}\Em{\mu_0}{\left(\sqrt{\Em{\nu_h}{\expapp(-\Phi_h)}} - \sqrt{\expapp^M(-\Phi_h)}\right)^2},\\
		I_2 &\defeq Z_h^M\left(\frac1{\sqrt{\Em{\nu_h}{Z_h}}} - \frac1{Z_h^M}\right)^2.
	\end{aligned}
	\end{equation}
	We first consider $I_1$. Proceeding as in the proof of \cref{thm:ConvMarginal}, we obtain
	\begin{equation}
	\begin{aligned}
		\Em{\nu_h}{Z_h}\Em{\nu_h}{I_1} &\leq 
		\begin{aligned}[t]
			&\Em{\nu_h}{\Em{\mu_0}{\left(\Em{\nu_h}{\expapp(-\Phi_h)} - \expapp^M(-\Phi_h)\right)^2}}\\
			&\qquad \quad \times \norm{\norm{\left(\Em{\nu_h}{\expapp\left(-\Phi_h\right)} + \expapp^M(-\Phi_h)\right)^{-1}}_{L^\infty_{\mu_0}(X)}}_{L^\infty_{\nu_h}(\Omega)},
		\end{aligned}
	\end{aligned}
	\end{equation}
	where moreover it holds for the second factor on the right-hand side
	\begin{equation}
	\begin{aligned}
		&\norm{\left(\Em{\nu_h}{\expapp\left(-\Phi_h\right)} + \expapp^M(-\Phi_h)\right)^{-1}}_{L^\infty_{\mu_0}(X)} \\
		&\qquad \qquad \qquad \leq \min\left\{\norm{\Em{\nu_h}{\expapp(-\Phi_h)}^{-1}}_{L^\infty_{\mu_0}}, \norm{\expapp^M(-\Phi_h)^{-1}}_{L^\infty_{\mu_0}}\right\}, \qquad \nu_h\text{-a.s.}
	\end{aligned}
	\end{equation}
	We now remark that applying Jensen's and the discrete Jensen inequality yield under \cref{as:Marginal}
	\begin{equation}
	\begin{aligned}
		&\Em{\nu_h}{\expapp(-\Phi_h)}^{-1} \leq C \Em{\nu_h}{\exp(\Phi_h)},\\
		&\expapp^M(-\Phi_h)^{-1} \leq C \exp^M(\Phi_h), \qquad \nu_h\text{-a.s.},
	\end{aligned}
	\end{equation}
	where $C$ is a positive constant independent of $h$ and $M$. Hence, it holds under \cref{as:Marginal}
	\begin{equation}
		\norm{\left(\Em{\nu_h}{\expapp\left(-\Phi_h\right)} + \expapp^M(-\Phi_h)\right)^{-1}}_{L^\infty_{\mu_0}(X)} \leq C, \qquad \nu_h\text{-a.s.},
	\end{equation}
	and therefore
	\begin{equation}
		\Em{\nu_h}{I_1} \leq C\Em{\nu_h}{\Em{\mu_0}{\left(\Em{\nu_h}{\expapp(-\Phi_h)} - \expapp^M(-\Phi_h)\right)^2}}.
	\end{equation}
	We then apply Fubini's theorem and \cref{prop:MSEExpApp} to obtain
	\begin{equation}\label{eq:ProofSam_I1}
		\Em{\nu_h}{I_1} \leq C \Em{\mu_0}{\Em{\nu_h}{\left(\Em{\nu_h}{\expapp(-\Phi_h)} - \expapp^M(-\Phi_h)\right)^2}} \leq C \frac{h^{2s}}{M}.
	\end{equation}
	We now consider $I_2$ and remark that it is equal to the quantity with the same symbol in the proof of \cref{thm:ConvMarginal}. Therefore
	\begin{equation}
		\Em{\nu_h}{I_2} \leq C \frac{h^{2s}}{M},
	\end{equation}
	which, together with \eqref{eq:ProofSam_I1I2} and \eqref{eq:ProofSam_I1}, yields the desired result.
\end{proof}

\section{Conclusion}\label{sec:Conclusion}

We presented and analyzed techniques for approximating the solution of inverse problems involving randomized approximations of the forward map. In particular, we rigorously studied the convergence of sampling-based posterior measures to the intractable marginal and sample approximations of the true posterior which have been recently introduced in \cite{LST18}. Our analysis shows that in both cases the number of samples which is needed in order to obtain a good approximation of the posterior could be set to small number, in case the discretization parameter is small. Moreover, we described and compared numerically MCMC techniques that allow to sample from the posteriors which we introduced and practically solve the inverse problem.

\subsection*{Acknowledgments}

The author is partially supported by the Swiss National Science Foundation, under grant No. 200020\_172710. The author thanks T.J. Sullivan for interesting advice on the topic of this work, and Assyr Abdulle for invaluable scientific and personal support over the years. 

\def\cprime{$'$}

\end{document}